\documentclass[11pt]{amsart}
\usepackage[color=blue!20!white,textsize=tiny]{todonotes}

\usepackage{amsmath} 
\usepackage{amsthm}
\usepackage{amssymb} 
\usepackage{amsfonts}
\usepackage{tikz-cd}
\usepackage{graphicx}
\usepackage{graphics}

\usepackage{mathtools}
\usepackage{comment}
\usepackage{tikz}
\usepackage{subfigure}
\usepackage[pagebackref=true]{hyperref}
\usepackage{soul}
\usepackage{float}

\newtheorem{theorem}{Theorem}[section]
\newtheorem{lemma}[theorem]{Lemma}
\newtheorem{proposition}[theorem]{Proposition}
\newtheorem{question}[theorem]{Question}

\newtheorem{corollary}[theorem]{Corollary}

\theoremstyle{definition}
\newtheorem{definition}[theorem]{Definition}
\newtheorem{example}[theorem]{Example}

\newcommand{\into}{\hookrightarrow}

\newcommand{\gru}{\mathrm{gr}_U}
\newcommand{\grv}{\mathrm{gr}_V}

\newcommand{\D}{\mathbb{D}}
\newcommand{\F}{\mathbb{F}}

\newcommand{\R}{\mathbb{R}}
\newcommand{\Z}{\mathbb{Z}}

\newcommand{\cF}{\mathcal{F}}
\newcommand{\cH}{\mathcal{H}}

\newcommand{\cR}{\mathcal{R}}

\newcommand{\cZ}{\mathcal{Z}}

\newcommand\alphas{\boldsymbol\alpha}
\newcommand\betas{\boldsymbol\beta}

\newcommand{\spinc}{\mathrm{Spin}^{c}}
\newcommand{\Sg}{\Sigma_{g}}

\newcommand{\Ta}{\mathbb{T}_{\alpha}}
\newcommand{\Tb}{\mathbb{T}_{\beta}}

\newcommand{\x}{\mathbf{x}}
\newcommand{\y}{\mathbf{y}}

\def\s{\mathfrak s}

\def\CF {\mathit{CF}}
\def\HF {\mathit{HF}}
\newcommand \HFhat{\widehat{\HF}}
\newcommand \CFhat{\widehat{\CF}}

\newcommand{\CFKm}{CFK^{-}}

\newcommand{\CFKhat}{\widehat{CFK}}

\newcommand{\HFKhat}{\widehat{HFK}}

\newcommand{\CFKR}{CFK_{\cR}}

\newcommand{\Hee}{\mathcal{H}}

\newcommand{\Sbar}{\overline{\Sigma}_{g}}
\newcommand{\muk}{K^{*}_{1/n}}

\newcommand{\Dmu}{D_{+}(K^{*}_{1/n})}
\newcommand{\dW}{S^{3}_{1/n}(K)}
\newcommand{\dWxI}{S^{3}_{1/n}(K)\times I}

\newcommand{\del}{\partial}
\newcommand \sthree {S^{3}}

\title{Deeply slice knot detection via immersed curves}

\author[R. McConkey]{Rob McConkey}
\address{Department of Mathematics and Physics, Colorado State University Pueblo, 2200 Bonforte Blvd, Pueblo, CO 81001}
\email{rob.mcconkey@csupueblo.edu}

\author[C. St. Clair]{Christopher St. Clair}
\address {Department of Mathematics, Saginaw Valley State University, University Center, MI 48710}
\email{cdstclai@svsu.edu}

\author[T. Wells]{Tristan Wells}
\address {Department of Mathematics, Louisiana State University, Baton Rouge, LA 70803}
\email{trwells22@lsu.edu}

\author[C. Zhang]{Chen Zhang}
\address {Simons Center for Geometry and Physics, State University of New York, Stony Brook, NY 11794}
\email{czhang@scgp.stonybrook.edu}

\begin{document}
\vspace*{-1cm}
\maketitle
\vspace*{-0.4cm}

\begin{abstract}
On the Kirby list, Akbulut poses the question of whether there exists a homology 3-sphere $Y$, other than $\sthree$, with the following property: Any knot $K$, representing $0\in\pi_{1}(Y),$ which is slice in some contractible 4-manifold $X$ which $Y$ bounds, is already slice in $Y\times[0,1]$. In this paper, we make progress on this question by producing a class of deeply slice knots. We construct these knots by first specifying a pair $(X, K)$, where $X$ is a contractible 4-manifold with integral homology 3-sphere boundary and $K$ is slice in $X$. Then, we show the knot is deeply slice using concordance invariants from Heegaard Floer homology. We employ immersed curve techniques to compute these invariants.
\end{abstract}

\tableofcontents

\section{Introduction}

\textit{Deeply slice} knots first appeared in \cite{klug2021deep} as part of the search for pairs of exotic 4-manifolds for the purposes of attacking the Smooth 4-Dimensional Poincar\'{e} Conjecture, which claims that every smooth 4-manifold that is homotopy equivalent to the 4-sphere is also diffeomorphic to the standard 4-sphere. In particular the 4D Poincar\'{e} conjecture would be false if there was a knot $K$ in $\sthree$, where $\sthree$ is the boundary of some exotic homotopy 4-ball $X$ such that $K$ is \textit{slice} in $X$ but the requisite smoothly embedded disk $\D\into X$ whose boundary is $K$ cannot be found in a collar neighborhood $\del X \times I$. Recall that if $X$ is a smooth 4-manifold with boundary $Y$, we say a knot $K\subset Y$ is \textit{smoothly slice}, or simply \textit{slice in $X$}, if there exists a smoothly embedded disk $\D\into X$ such that $K=\del\D\subset\del X=Y.$ Equivalently, a knot is slice if it is \textit{concordant} to the unknot. A schematic of a slice disk as well as a trivial example is shown in Figure \ref{fig:sliceexample}. Sliceness is an extremely active topic for study when considering knots in $\sthree$ thought of as the boundary of the standard smooth 4-ball, $B^{4}$. When not specified, when we say ``$K$ is slice,'' we mean that $K\subset \sthree$ is slice in $B^{4}$. Klug and Ruppik introduced a more nuanced property called \textit{deeply slice} in \cite{klug2021deep}.

\begin{figure}
    \centering
    \includegraphics[width = .8\linewidth]{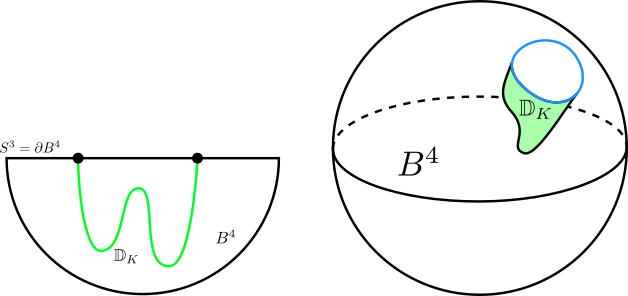}
    \caption{Left: A half-dimensional schematic of a slice disk for some knot $K$. Right: A slice disk schematic for the unknot, pushed into $B^{4}.$}
    \label{fig:sliceexample}
\end{figure}

\begin{definition}[Deeply Slice, \cite{klug2021deep}]\label{def:deeplyslice}
    A knot $K\subset\del X$ is \textit{deeply slice} in $X$ if it is slice in $X$ but $K$ is not slice in $\del X\times I$.
\end{definition}

In other words, none of the slice disks for $K$ live in $\del X \times I$ the collar neighborhood of the boundary. In essence, this means that the slice disk is ``interesting'' because it requires one to dig ``deep'' into the 4-manifold $X$. If a knot is slice but not deeply slice, we call it \textit{shallow slice}. 

While the Smooth 4-Dimensional Poincar\'{e} Conjecture motivated this notion, Klug and Ruppik used it to make progress on a question on the Kirby list attributed to Akbulut: 

\begin{question}[Kirby list 1.95, \cite{kirby1997problems}]\label{q:1.95}
    Does there exist a homology 3-sphere $Y$, other than $\sthree$, with the following property: Any knot $K$, representing $0\in\pi_{1}(Y),$ which is slice in some contractible 4-manifold $X$ which $Y$ bounds, is already slice in $Y\times[0,1]$?
\end{question} 

The negation of Question \ref{q:1.95} can be phrased as: Does every homology 3-sphere bound a contractible, smooth 4-manifold which contains a deeply slice knot that is null-homotopic in the boundary? The relevant result to this paper is the following theorem.

\begin{theorem}[\cite{klug2021deep}]
    Every 2-handlebody $X$ contains a null-homotopic deeply slice knot in its boundary.
\end{theorem}

Instead of starting with candidate 4-manifolds and seeking deeply slice knots, our main theorem starts with a slice knot and constructs a 4-manifold:

\begin{theorem}
    Associated to a non-trivial slice knot (with some constraints) and any $n\in \Z_{>0}$, there is a contractible, smooth 4-manifold that contains a deeply slice knot which is null-homotopic in its boundary.
\end{theorem}

See Theorem \ref{thm:deepslice} for a technical version, which imposes some further restrictions on the slice knot. The proof of the main theorem has two steps. First, we construct a 4-manifold and a knot with the desired properties, then show that said knot is not slice in a collar neighborhood of the boundary. The novel component to the proof is in the second step, in which we utilize the $\tau_{\alpha}$-invariants developed in \cite{heddenraoux}, generalizing the $\tau$-invariant coming from knot Floer homology. Similar to the $\tau$-invariant, these invariants obstruct sliceness. Geometrically, $\tau_{\alpha}(Y,K)$ bounds the genus of surfaces with boundary $K$ in 4-manifolds with boundary $Y$. In order to compute these, we turn to more recent technology in \textit{bordered Floer theory}. In particular, the \textit{immersed curves} perspective on bordered invariants, developed in \cite{hanselman2023bordered} serves as the perfect visual tool to compute $\tau_{\alpha}$. In this paper, we refine a \textit{pairing theorem} from \cite{chenhanselman} to fully leverage their \textit{immersed Heegaard Floer homology}.

The magic of obstructing sliceness using $\tau_{\alpha}$ is that the invariant depends entirely on the knot and the 3-manifold. Therefore, we have agency in constructing a suitable 4-manifold with the needed topological properties like contractibility. We use this freedom along with a result of Gordon to construct a slice knot in a contractible 4-manifold with boundary an integral homology 3-sphere. This differs from the approach of Klug and Ruppik, who had to operate in a setting with prescribed 4-manifolds (such as 2-handlebodies) and then seek knots with the right sliceness properties. 

Ultimately, the 3-manifold we consider is the $1/n$-surgery (described below) on the provided slice knot, and the knot we consider is the Whitehead double of the dual knot to that surgery. Performing $1/n$-surgery guarantees that the resulting 3-manifold is an integral homology 3-sphere. We choose the 4-manifold so that the dual knot is slice in it, and Whitehead doubling preserves sliceness. The tricky component to the proof is ensuring that the immersed curves for the class of slice knots we consider behave appropriately, which boils down to a diagrammatic approach. 

\subsection{Organization}

Section \ref{section:topology} details the construction of the 4-manifold in the main theorem. Here, we also give Gordon's result and provide the slice knot in consideration for being deeply slice. In Section \ref{chap:HF}, we define $\tau_{\alpha}(Y,K)$. In Section \ref{sec:bordered} we introduce immersed curves, immersed Heegaard Floer theory, and our refined pairing theorems for the context of this paper. At the end of Section \ref{sec:bordered}, we provide example computations of immersed Heegaard diagrams and extract invariants from them. In Section \ref{sec:maintheorem}, we prove Theorem \ref{thm:deepslice}, the technical version of the main theorem.

\subsection{Acknowledgments} We thank Matthew Hedden for inspiration on the project as well as guidance in learning required skills. We thank Whenzhao Chen, Jake Rasmussen, Luke Seaton, Ivan So, Dean Spyropoulos, and Rithwik Vidyarthi for helpful conversations.

We also want to thank the National Science Foundation for supporting our research with grants DMS-2104664 supporting St. Clair, Wells, and Zhang, the RTG NSF Award DMS-2135960 supporting both McConkey and Wells who was also supported by DMS-1709016. And DMS-2304033 which provided partial support for McConkey. Their financial support made this research possible.

\section{Topological Preliminaries}\label{section:topology}

Dehn surgery and satellite knots are topological tools used throughout the paper that we review here.

\textit{Dehn surgery} is a method to construct a new 3-manifold from a given 3-manifold by cutting and pasting along a knot. In this paper, we focus on the case for knots in $S^3$. Formally, one excises a neighborhood of the given knot $K$, whose closure is homeomorphic to a solid torus, to obtain the knot complement, then reattaches the solid torus to the knot complement along their common boundary, a torus, via some homeomorphism of the torus $h:\sthree\setminus\nu(K)\to\del\,\overline{\nu(K)}.$ The result is a new, closed 3-manifold $\sthree_{p/q}(K):=\sthree\setminus\nu(K)\cup_{h}\del\,\overline{\nu(K)}$. Note that, since $K$ is in $S^3$, there is a natural choice of the meridian and longitude of the boundary of the knot complement, which we denote by $(\mu,\lambda)$. The pair $(p,q)$ describes the image of the meridian of the solid torus as $p\mu+q\lambda$, which uniquely determines the homeomorphism $h$.

A \textit{satellite operation} is a method of constructing a new knot from two given knots. Let $P$ be a knot embedded in the solid torus $S^{1}\times\D^{2}$ and $K$ an arbitrary knot in $\sthree$. Gluing $S^{1}\times\D^{2}$ to $\sthree\setminus\nu(K)$ along their common torus boundary yields a new knot in $S^{3}$, where the meridians of the tori are identified to each other and the longitude of $S^{1}\times\D^{2}$ is identified with the Seifert longitude of $K$. The resulting knot is called the \textit{satellite knot} $P(K)$, which is the image of $P$ under the gluing.  We call $P\subset S^{1}\times\D^{2}$ the \textit{pattern knot} and $K\subset\sthree$ the \textit{companion knot}. We pay special attention to the Whitehead pattern, denoted $D_{+}$. Here, the Whitehead pattern is the positively clasped unknot wrapped around the $S^{1}$ factor of the solid torus in which its embedded. See Figure \ref{fig:satelliteop} for the Whitehead double of the right-handed trefoil.

\begin{figure}
    \centering
    \includegraphics[width = 4in]{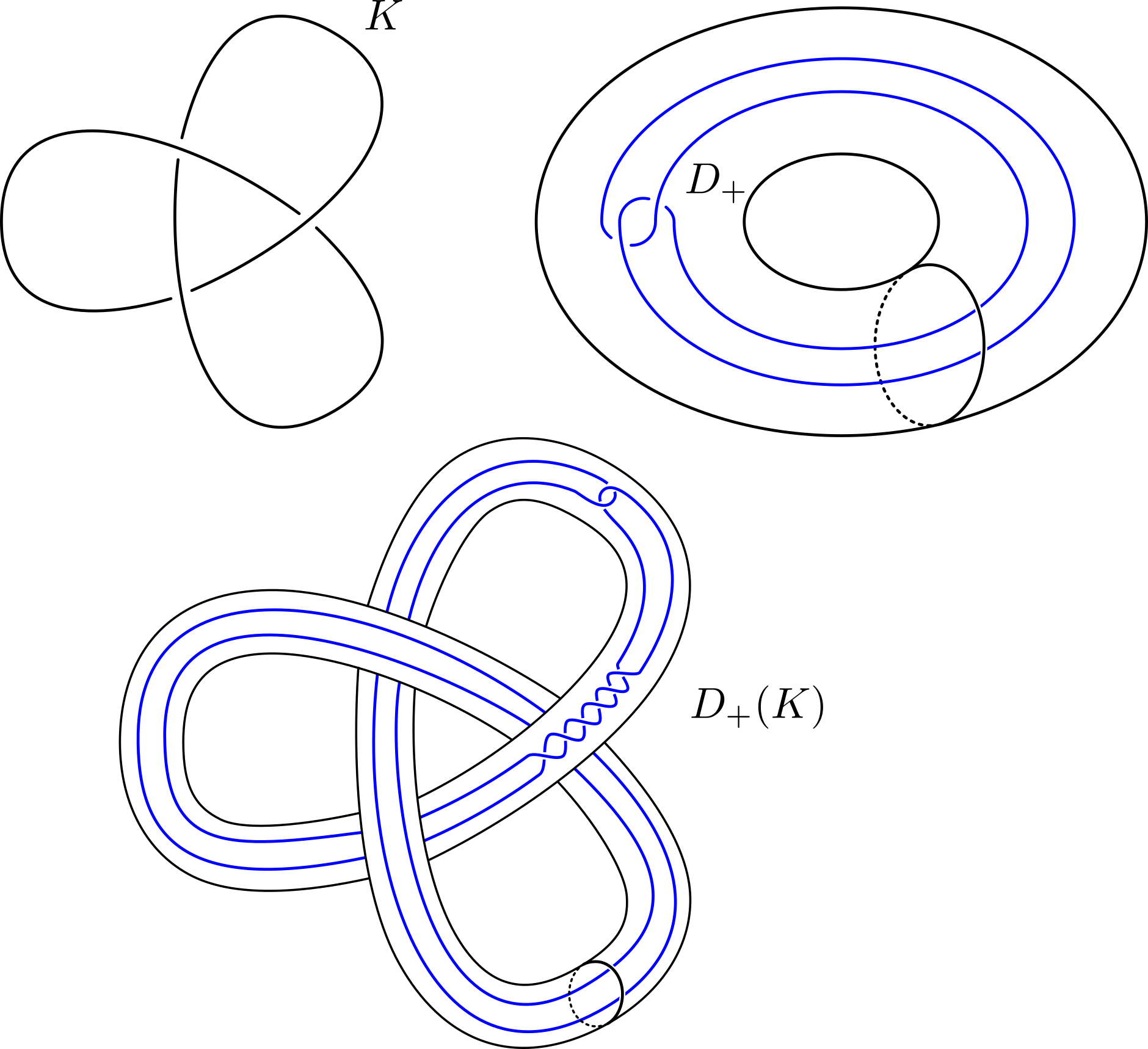}
    \caption{Left: A diagram for the right handed trefoil in $\sthree$. Right: The Whitehead pattern, $D_{+}$, in $S^{1}\times\D^{2}$. Bottom: A diagram for the knot $D_{+}(T_{2,3})$ in $\sthree$.}
    \label{fig:satelliteop}
\end{figure}

Recall that surgery along $K$ requires removing a solid torus in $\sthree$ and gluing it back in. The image of the core of the surgery solid torus under gluing is called the \textit{dual knot} to the surgery along $K$. In the case of $1/n$ surgery, we denote the dual knot as $\muk\subset\dW$. Much information is known about dual knots, as well as various flavors of their knot Floer homology, including surgery formulae in \cite{hedden2019dual}.

We highlight a feature of the dual knot of an integral surgery here. 
Given an integral framed link $(L,\Lambda)$ in $S^3$, one can construct a 4-manifold by attaching 2-handles along the framed link to $B^4$. The resulting 4-manifold has $\sthree_{\Lambda}(L)$ as the boundary. See Figure \ref{fig:handleattach} for an example, where the 2-handle is parametrized as $\D^{2}\times \D^{2}.$ Therefore, each core of the surgery tori is isotopic to the boundary of the cocore (the second $\D^{2}$ factor) of the attached 2-handle. Thus, the dual knot is slice in the 4-manifold, where the slice disk is the cocore of the 2-handle itself.  For a reference on handle decompositions of 4-manifolds and Kirby calculus, see \cite{gompf19994}. Moreover, 

\begin{theorem}[\cite{gordon1975knots}]\label{thm:contractible}
    If $K$ is slice, then there is a smooth, contractible 4-manifold $W$ with boundary $\dW$.
\end{theorem}

\begin{figure}
    \centering
    \includegraphics{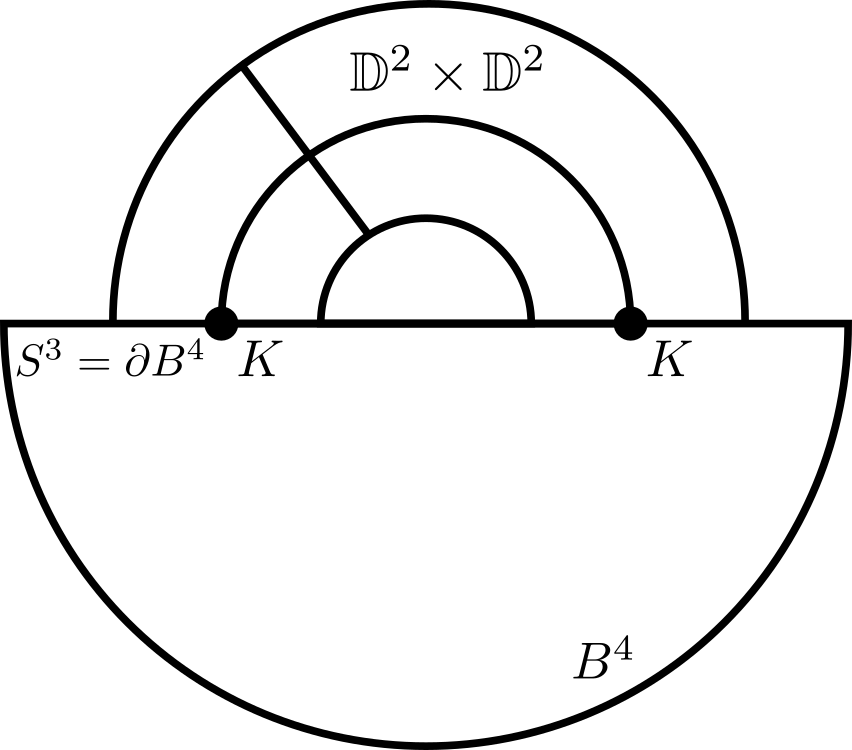}
    \caption{A half-dimensional schematic if a 2-handle attachment.}
    \label{fig:handleattach}
\end{figure}

There is a nice Kirby diagram for a 4-manifold $W$ as in Theorem~\ref{thm:contractible} obtained from the surgery diagram for $\dW$, shown in Figure \ref{fig:slam dunk}. We first note that the dual knot to $1/n$ surgery along $K$ looks like a pushoff of $K$ in the surgery diagram in Figure \ref{fig:1/n surgery}. We then perform a Kirby move known as the \textit{slam dunk} to obtain a surgery diagram with 0-framed surgery on $K$ and $-n$-framed surgery on a new, unknotted component, denoted $L$. Finally, we slide the dual knot representative over $K$ itself to obtain the desired surgery diagram for $\dW$. Now we show that all the Kirby diagrams in Figure \ref{fig:handleslide} are Kirby diagrams for $\dW$.

\begin{figure}
    \centering
    \includegraphics[scale=.7]{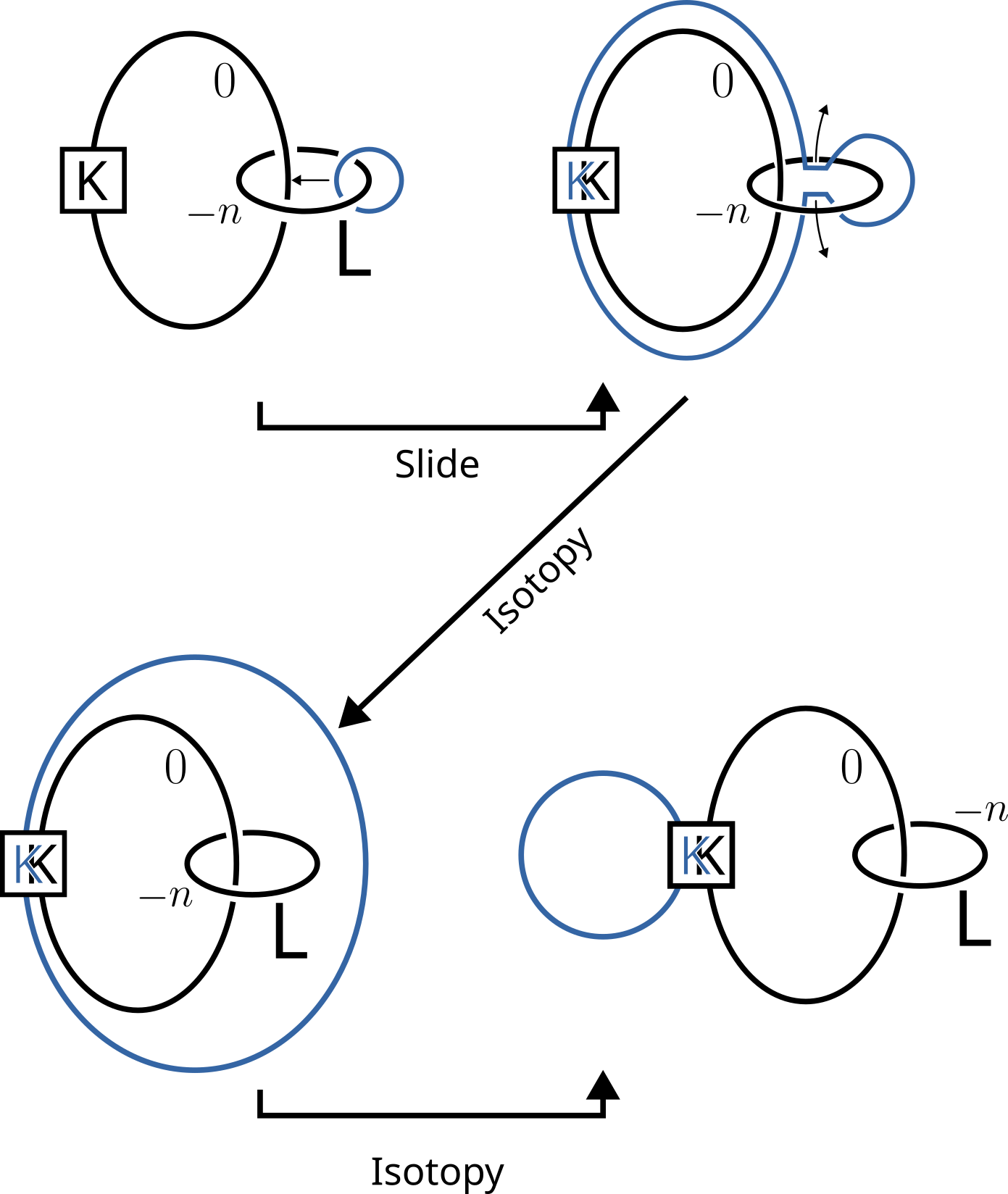}
    \caption{Using isotopies and handle slides to see the isotopy class of the dual knot to $-n$ surgery on $L$.}
    \label{fig:handleslide}
\end{figure}

\begin{figure}
    \centering
    \includegraphics{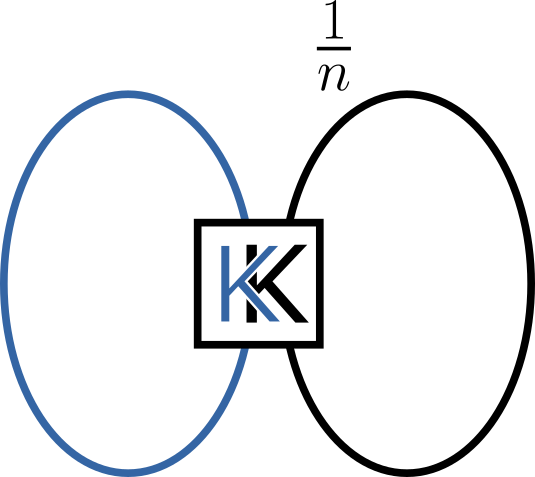}
    \caption{A surgery diagram for $1/n$ surgery along $K$, keeping track of the dual knot to the surgery, $\muk$, in blue.}
    \label{fig:1/n surgery}
\end{figure}

\begin{figure}
    \centering
    \includegraphics[width=.95\linewidth]{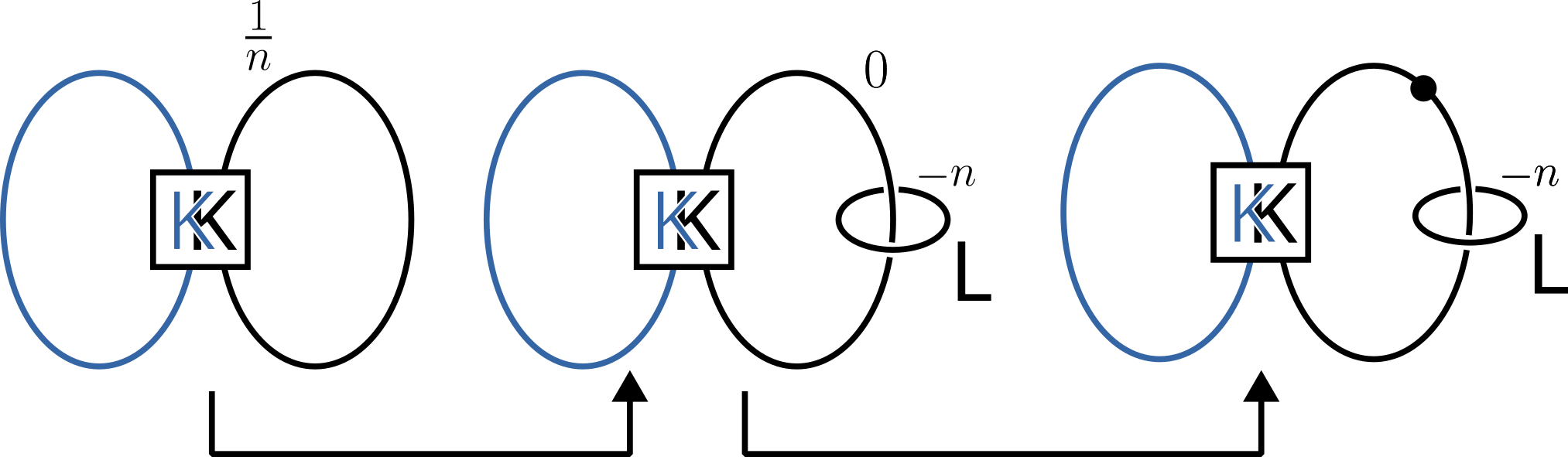}
    \caption{A \textit{slam dunk} Kirby move applied to obtain an integral surgery diagram for $\dW$. On the right is a Kirby diagram for a new 4-manifold, $W$, whose boundary is homeomorphic to $\dW$.}
    \label{fig:slam dunk}
\end{figure}

\begin{proposition}\label{prop:W}
    Suppose $K$ is slice. Then the Kirby diagram on the right hand side of Figure \ref{fig:slam dunk} describes a smooth, contractible 4-manifold $W$ whose boundary is $\dW$.
\end{proposition}

\begin{proof}
    First, we note that the sequence of Kirby diagrams in Figure \ref{fig:slam dunk} are surgery diagrams for $\dW$, by Kirby's Theorem \cite{kirby1978calculus}. This is because we used only Kirby moves to manipulate the diagram, keeping track of the isotopy class of the dual knot, in blue. Since $K$ is slice in $\sthree=\del B^{4}$, we begin by removing a neighborhood of a slice disk for $K$. This neighborhood is diffeomorphic to $\D^{2}\times \D^{2}$, with the first factor thought of as the slice disk, and the second as the thickening to a neighborhood. Hence, we are removing a 2-handle from $B^{4}$ by instead introducing its canceling 1-handle, as in the ``digging a ditch'' analogy in \cite{gompf19994}. With this perspective, we can see that the fundamental group of $B^{4}\setminus (\D^{2}\times\D^{2})$ is generated by a meridian of the boundary of the removed disk; that is, a meridian of $K$. Now, we attach the other 2-handle with framing $-n$ along $L$, which is also a meridian of $K$, killing that generator in $\pi_{1}(B^{4}\setminus(\D^{2}\times\D^{2}))$, so $W=(B^{4}\setminus (\D^{2}\times\D^{2}))\cup (\mathrm{2-handle})$ is contractible.
\end{proof}

The particular feature of the dual knot to $1/n$ surgery which is important to us is the following.

\begin{lemma}\label{lem:dualslice}
    The dual knot to $1/n$ surgery along a slice knot $K\subset\sthree$, $\muk\subset\dW$, is slice in $W$ as in Proposition \ref{prop:W}.
\end{lemma}

\begin{proof}
    This can readily be seen using the argument above, since the blue curve is isotopic to the dual knot to $-n$ surgery on $L$ in Figure \ref{fig:slam dunk}. Thus, it is the boundary of the core of the 2-handle attached along $L$, which is a smooth disk in $W$.
\end{proof}

A well known fact is that the Whitehead double of a smoothly slice knot in $\sthree$ is again smoothly slice. Analogously, we achieve the following corollary:

\begin{corollary}\label{cor:dualslice}
    The Whitehead double of the dual knot to $1/n$ surgery along a slice knot $K\subset\sthree$, $\Dmu$, is slice in $W$ as in Proposition \ref{prop:W}.
\end{corollary}

\begin{proof}
    By Lemma \ref{lem:dualslice}, the dual knot is itself slice. Given the dual knot is slice we can find an annulus with one boundary component the dual knot and one boundary the unknot, then we can ``Whitehead double'' the annulus in the same fashion as the dual knot itself. Since the Whitehead double of the unknot is again the unknot, we have presented an annulus between the Whitehead double of the dual knot and the unknot.
\end{proof}

\section{Invariants}\label{chap:HF}

The invariants that we use to obstruct shallow sliceness are from Heegaard Floer homology. Since the introduction of Heegaard Floer homology by Ozsv\'{a}th and Szab\'{o} in 2001 \cite{ozsvath2004heegaard}, many versions and refinements have arisen to accommodate various settings in low-dimensional topology and knot theory. While the original definition assigns a graded vector space to a closed, oriented, smooth 3-manifold equipped with a $\spinc$ structure, Lipschitz, Ozsv\'{a}th, and Thurston have since adapted the tool to bordered 3-manifolds in \cite{lipshitz2018bordered}, calling it Bordered Heegaard Floer homology. The power of bordered Heegaard Floer homology is the so-called pairing theorems, which give a way to compute the Heegaard Floer homology groups of closed 3-manifolds obtained by gluing two bordered 3-manifolds along their common boundary in a specified way. Dehn surgery and satellites are two prominent examples of such gluings that arise naturally in the context of a knot in a closed 3-manifold $Y$.

Here, we briefly recall the definition of the hat version of Heegaard Floer homology for both 3-manifolds and knots. For further details, see \cite{ozsvath2004heegaard,ozsvath2004knots}. Let $\F=\Z/2\Z$.
\begin{definition}\label{def:CFinf}
    Given a closed, smooth 3-manifold $Y$ together with a $\spinc$ structure $\s$ and a Heegaard diagram $(\Sg,\alphas,\betas,w)$ for $Y$, the \textit{Heegaard Floer chain complex} of $(Y,\s)$ is freely generated over $\F$ by all intersection points $\x\in\Ta\cap\Tb$ such that $\s_{w}(\x)=\s$, denoted $$\CFhat(Y,\s):=\bigoplus_{\substack{\x\in\Ta\cap\Tb\\ \s_{w}(\x)=\s}}\F,$$ with boundary map defined on generators by $$\partial(\x):=\sum_{y\in \mathbb{T}_{\alpha}\bigcap\mathbb{T}_{\beta}}\sum_{\substack{ {\substack{\phi\in\pi_{2}(x,y)\\ \mu(\phi)=1}} \\ n_{w}(\phi)=0 }}\#\widehat{\mathcal{M}}(\phi)\y.$$
\end{definition}

We denote the homology of the Heegaard Floer chain complex by $\HFhat(Y,\s)$. Note that, when the manifold is $S^{3}$, we have $\HFhat(S^{3})\cong \F$.

\begin{definition}\label{def:CFKUV}
    Given a knot $K$ in $Y$ and a doubly-pointed Heegaard diagram $\cH=(\Sg,\alphas,\betas,w, z)$ compatible with $K,$ the \textit{hat knot Floer complex} of the pair $(Y, K)$ is freely generated over $\F$ by all intersection points $\x\in\Ta\cap\Tb$, denoted $$\CFKhat(Y, K):=\bigoplus_{\x\in\Ta\cap\Tb}\F,$$ with boundary map defined on generators by $$\partial(\x):=\sum_{y\in \mathbb{T}_{\alpha}\bigcap\mathbb{T}_{\beta}}\sum_{\substack{{\substack{\phi\in\pi_{2}(x,y)\\ \mu(\phi)=1}}\\n_{w}(\phi)=0, n_{z}(\phi)=0} }\#\widehat{\mathcal{M}}(\phi)\y.$$
\end{definition}

In addition to the hat flavor of the knot Floer complex, we employ the \textit{minus} flavor in order to gain extra structure, namely a module structure coming from an action by a formal variable $U.$

\begin{definition}\label{def:CFKm}
    Given a knot $K$ in $Y$ and a doubly-pointed Heegaard diagram $\cH=(\Sg,\alphas,\betas,w, z)$ compatible with $K,$ the \textit{minus knot Floer complex} of the pair $(Y, K)$ is freely generated over $\F[U]$ by all intersection points $\x\in\Ta\cap\Tb$, denoted $$\CFKm(Y, K):=\bigoplus_{\x\in\Ta\cap\Tb}\F[U],$$ with boundary map defined on generators by $$\partial^{-}(\x):=\sum_{y\in \mathbb{T}_{\alpha}\bigcap\mathbb{T}_{\beta}}\sum_{\substack{{\substack{\phi\in\pi_{2}(x,y)\\ \mu(\phi)=1}}}}\#\widehat{\mathcal{M}}(\phi)U^{n_{w}(\phi)}\y.$$
\end{definition}

When the manifold in question is $\sthree$, we simply write $\CFKm(K)$ and omit reference to the manifold. For distinction between the complexes $\CFKhat(T_{2,3})$ and $\CFKm(T_{2,3})$, compare Figures \ref{algebraictau} and \ref{fig:immersedexample}.

The relative Alexander grading, induced by a knot $K$, is calculated by taking a holomorphic disk connecting a pair of generators $(\x,\y)$, i.e. $\phi\in\pi_{2}(\x,\y)$, using the relation
\begin{center}
    $A(\x)-A(\y)=n_{z}(\phi)-n_{w}(\phi)$.
\end{center}
One can lift it to an absolute grading using the symmetrized Alexander polynomial. Using the filtration induced by the Alexander grading, Ozsv\'{a}th and Szab\'{o} defined the $\tau$-invariant for a knot $K$ in $S^{3}$ as follows. Let $\cF(K,m)\subset\CFhat(S^{3})$ be the subcomplex generated by intersection points whose filtration level is less than or equal to $m$. We obtain an induced sequence of maps

\begin{center}
    $\iota_{K}^{m}: H_{\ast}(\cF(K,m))\rightarrow H_{\ast}(\CFhat(S^{3}))=\HFhat(S^{3})\cong \F$,
\end{center}
which for sufficiently large $m$ are all isomorphisms. Then
\begin{center}
    $\tau(K)=\mathrm{min}\{m\in\Z|\iota_{K}^{m} \text{ is non-trivial}\}$.
\end{center}
Equivalently, one can use the spectral sequence induced by the Alexander filtration. The first page corresponds to the knot Floer chain complex of the knot, and the infinity page computes the Heegaard Floer homology of the ambient manifold $S^3$, whose homology is rank 1. The $\tau$-invariant can be defined as the Alexander grading of the unique cycle on the infinity page.

\begin{example}
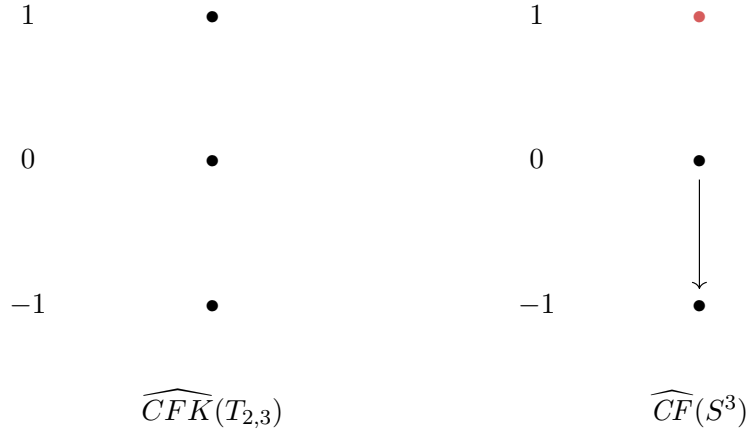
\begin{figure}
    \centering
   \[\begin{tikzcd}
	1 & \bullet &&& 1 & \textcolor{rgb,255:red,214;green,92;blue,92}{\bullet} \\
	\\
	0 & \bullet &&& 0 & \bullet \\
	\\
	{-1} & \bullet &&& {-1} & \bullet \\
	& {\CFKhat(T_{2,3})} &&&& {\CFhat(S^{3})}
	\arrow[from=3-6, to=5-6]
\end{tikzcd}\]
    \caption{The vertical level shows the Alexander grading of the corresponding generator, so $\tau(T_{2,3})=1.$}
    \label{algebraictau}
\end{figure}
 
In Figure \ref{algebraictau}, we have the complex for $T_{2,3}$ and $S^{3}$. The red dot is the one which remains nontrivial on the infinity page. Hence, we have $\tau(T_{2,3})=1$.

\end{example}

The $\tau$-invariant is generalized by Hedden and Raoux for a null-homologus knot $K$ in any closed, oriented three manifold $Y$. For simplicity, we will assume $Y$ to be an integer homology sphere.

\begin{definition}\cite{heddenraoux}
    For a nontrivial class $\alpha$ in $\HFhat(Y)$,
    \begin{center}
        $\tau_{\alpha}(Y,K)=\mathrm{min}\{m\in\Z|\alpha\in\mathrm{Im}(\iota_{K}^{m})\}$.
    \end{center}
\end{definition}

Geometrically, $\tau_{\alpha}(Y,K)$ gives a lower bound for the genus of surfaces with boundary $K$ in a 4-manifold with boundary.

\begin{proposition}[\cite{heddenraoux}]\label{prop:talpha}
    Let $K$ be a null-homologous knot in $Y$. If $\Sigma\subset Y\times I$ is a smoothly embedded oriented surface with boundary $K\subset Y\times \{1\}$, then $$|\tau_{\alpha}(Y,K)|\leq g(\Sigma).$$
\end{proposition}

\begin{corollary}\label{cor:taualpha}
    If $K\subset Y$ is slice in $Y\times I$, then $\tau_{\alpha}(Y,K)=0$ for all $\alpha\in \HFhat(Y)$.
\end{corollary}

\begin{proof}
    Recall that if $K$ is slice, then it bounds a smooth disk in $Y\times I$. If there exists some $\alpha\in\HFhat(Y)$ with $\tau_{\alpha}(K)\neq 0$, Proposition \ref{prop:talpha} implies that the genus of any such surface with boundary $K$ is nonzero, a contradiction.
\end{proof}

This is a generalization of the fact that the $\tau$-invariant obstructs sliceness in $B^{4}$, which means we can use the $\tau_{\alpha}$-invariants to obstruct shallow sliceness.

\section{Bordered Floer Theory}\label{sec:bordered}

The main tool used to compute the $\tau_\alpha$-invariant of the Whitehead double of the dual knot is \textit{bordered Heegaard Floer homology}. The most consolidated resource on the matter is \cite{lipshitz2018bordered}. Since its advent, bordered Heegaard Floer homology has led to significant results in 3-manifold topology and knot theory. In \cite{hanselman2023bordered,hanselman2022properties,hanselman2023immersed}, Hanselman, Rasmussen, and Watson use the framework of bordered Heegaard Floer homology to construct a very useful description of knot Floer homology using \textit{immersed curves}. Later, Chen and Hanselman developed immersed Heegaard Floer homology in \cite{chenhanselman}, which is a powerful tool to compute the knot Floer homology of satellite knots using immersed curves. Similar to standard Heegaard Floer homology, immersed Heegaard Floer homology involves Lagrangian intersection theory on a certain Heegaard diagram. The only difference is that the $\alpha$ curves may be immersed. For our results the immersed Heegaard diagrams are genus 1, the $\alpha$ curve corresponds to the immersed curve of the companion knot $K$, and the $\beta$ curve corresponds to the Heegaard diagram of the Whitehead pattern. In this section, we first introduce these two sets of curves used in immersed Heegaard Floer homology. Then we prove a suitable version of a pairing theorem to compute the knot Floer homology.

\subsection{Immersed Curves}\label{sec:immersedcurves}

The first piece in an immersed Heegaard diagram is the \textit{immersed multicurve} in the marked boundary of a knot compliment, the \textit{marked torus}. The marked torus is simply the standard torus, modeled as $\R^{2}/\Z^{2}$, where the $x$-axis is the \textit{preferred longitude} and the $y$-axis is the \textit{preferred meridian}, and a basepoint $z$ located at $(1-\epsilon, 1-\epsilon)$ for small $\epsilon>0$. An \textit{immersed multicurve} is a set of immersed curves in the marked torus away from $z$ decorated with \textit{local systems}, which we suppress, since our result does not concern them directly. We often refer to an immersed multicurve as an \textit{immersed curve}.

Immersed curves conveniently package the information of $\CFKm(Y,K)$ in the form of the \textit{bordered invariant}, sometimes denoted $\HFhat(Y\setminus\nu(K))$, which is an invariant of a bordered 3-manifold arising from a bordered Heegaard diagram. For discussions on the immersed curve formulation of bordered Heegaard Floer homology, see \cite{hanselman2023bordered}, \cite{hanselman2022properties}, and \cite{hanselman2023immersed}. Here, we only talk about the algorithm of getting the immersed curves directly from $\CFKm(Y,K)$, following the method in \cite{hanselman2022properties}.

To use this technique a particular basis is required for $\CFKm(Y,K)$, thought of as an $\F[U]$-module. Recall that $\CFKm(Y,K)$ comes with two filtrations, the action of multiplying by $U$ and the Alexander filtration. As in \cite{hanselman2022properties}, we may choose a representative of the chain homotopy type of $\CFKm(Y,K)$ for which the boundary map $\del^{-}$ strictly decreases one of these filtrations. A \textit{filtered basis} for $\CFKm(Y,K)$ is the set $\{v_{i}\}$ such that the equivalence classes $\{[v_{i}]\}$ in the associated graded complex $g\CFKm(Y,K)=\bigoplus \cF_{A(x)}/\cF_{A(x)-1}$ are a basis.

\begin{definition}[\cite{hanselman2022properties}]
    A filtered basis $\{v_{i}\}$ is \textit{vertically simplified} if for each $v_{i}$, either $\del v_{i}\in U\cdot\CFKm(Y,K)$ or $\del v_{i}\sim v_{j}+x$ where $x\in U\cdot \CFKm(Y,K)$. The filtered basis is \textit{horizontally simplified} if for each $v_{i}$ with Alexander filtration level $A(v_{i})=l$, either $A(\del v_{i})<l$ or $A(\del v_{i})=U^{k}v_{j}+x$ where $A(U^{k}v_{j})=l$ and $A(x)<l$.
\end{definition}

Being vertically simplified is equivalent to each basis generator only having one vertical arrow pointing to or away from it, and likewise for horizontally simplified. The complex in Figure \ref{fig:immersedexample} is both vertically and horizontally simplified. For any knot, $\CFKm(K)$ always admits a vertically simplified basis and always admits a horizontally simplified basis. As a warning, it may not be the case that $\CFKm(K)$ admits a basis which is simultaneously vertically and horizontally simplified. Nevertheless, we restrict our attention to when there is a basis which is both, in which case we just call the basis \textit{simplified}.

Following the method in \cite{hanselman2022properties}, we present the immersed curve representing $\CFKm(K)$ as a lift in the cover of the marked torus, $[-1/2,1/2]\times \R$, where the closed interval corresponds to the preferred longitude $\lambda$ and $\{-1/2\}\times \R$ and $\{1/2\}\times \R$ are identified and is a lift of the preferred meridian, $\mu$. We then center the lifts of the basepoint $z$ at $\{0\}\times \{n+1/2\}$ for $n\in\Z$. See Figure \ref{fig:immersedexample} for an example of this setup along with an example of a curve obtained from the procedure below.

\begin{proposition}[\cite{hanselman2022properties}]\label{prop:makethecurve}
    Given a horizontally and vertically simplified basis for \newline $\CFKm(K)$, a lift of the immersed multicurve $\alphas_{im}=\HFhat(S^3\setminus\nu(K))$ in the infinite strip can be obtained by the following procedure:
    \begin{enumerate}
        \item For each basis element $v_{i}$ of $\CFKm(K)$, place a short horizontal segment at \newline $[-1/4,1/4]\times\{t\}$ where $t=A(v_{i})$.
        \item If $\CFKm(K)$ contains a vertical arrow from $v_{i}$ to $v_{j}$, then connect the \textnormal{left} endpoints of the horizontal segments corresponding to $v_{i}$ and $v_{j}$ by a vertical arc.
        \item If $\CFKm(K)$ contains a horizontal arrow from $v_{i}$ to $v_{j}$, then connect the \textnormal{right} endpoints of the horizontal segments corresponding to $v_{i}$ and $v_{j}$ by a vertical arc.
        \item Connect the unique horizontal segment with an unattached left endpoint to $\{-1/2\}\times\{0\}$ and the unique horizontal line segment with an unattached right endpoint to $\{1/2\}\times\{0\}$ each with an arc.
    \end{enumerate}
\end{proposition}

We often denote the curve $\alphas_{im}$ as $\alphas_{K}$ when we wish to emphasize the knot. Given an immersed curve $\alphas_{K}$ coming from Proposition \ref{prop:makethecurve}, we can read off a few invariants. To do this, it is helpful to pull the immersed curve tight, to create a so-called \textit{pegboard diagram} as in Figure \ref{fig:pegboard}. For a knot, the immersed curve pulled tight will be vertical strands in the neighborhood of $\{0\}\times\R$, where the ``pegs'' are located, and one homologically horizontal strand which wraps around the infinite strip.

The genus of the knot, $g(K)$, is simply the difference between the maximum height achieved by $\alphas_{K}$ and the minimum height achieved, rounded to the nearest integer when pulled tight. This is the same as checking how many pegs are encompassed by the curve and dividing by 2. The $\tau$-invariant can be see by starting anywhere on the horizontal strand and tracing to the right (in the positive interval direction) until hitting the vertical strand. The nearest integer height where they meet is $\tau(K)$. Finally, $\epsilon(K)$, defined in \cite{hom2014bordered}, is shown by the behavior of the horizontal line segment after crossing the vertical portion. If the curve has an upward slope, $\epsilon(K)=1$. If downward, $\epsilon(K)=-1$. If it continues straight, which is only possible if $\tau(K)=0$ by Proposition \ref{prop:makethecurve}, then $\epsilon(K)=0$ as well. For example, from Figure \ref{fig:pegboard}, we see that $\tau(T_{2,3})=1$ and $\epsilon(T_{2,3})=1$, since the horizontal curve first intersects above the higher basepoint when traveling to the right (along the arrow indicating $\lambda$).

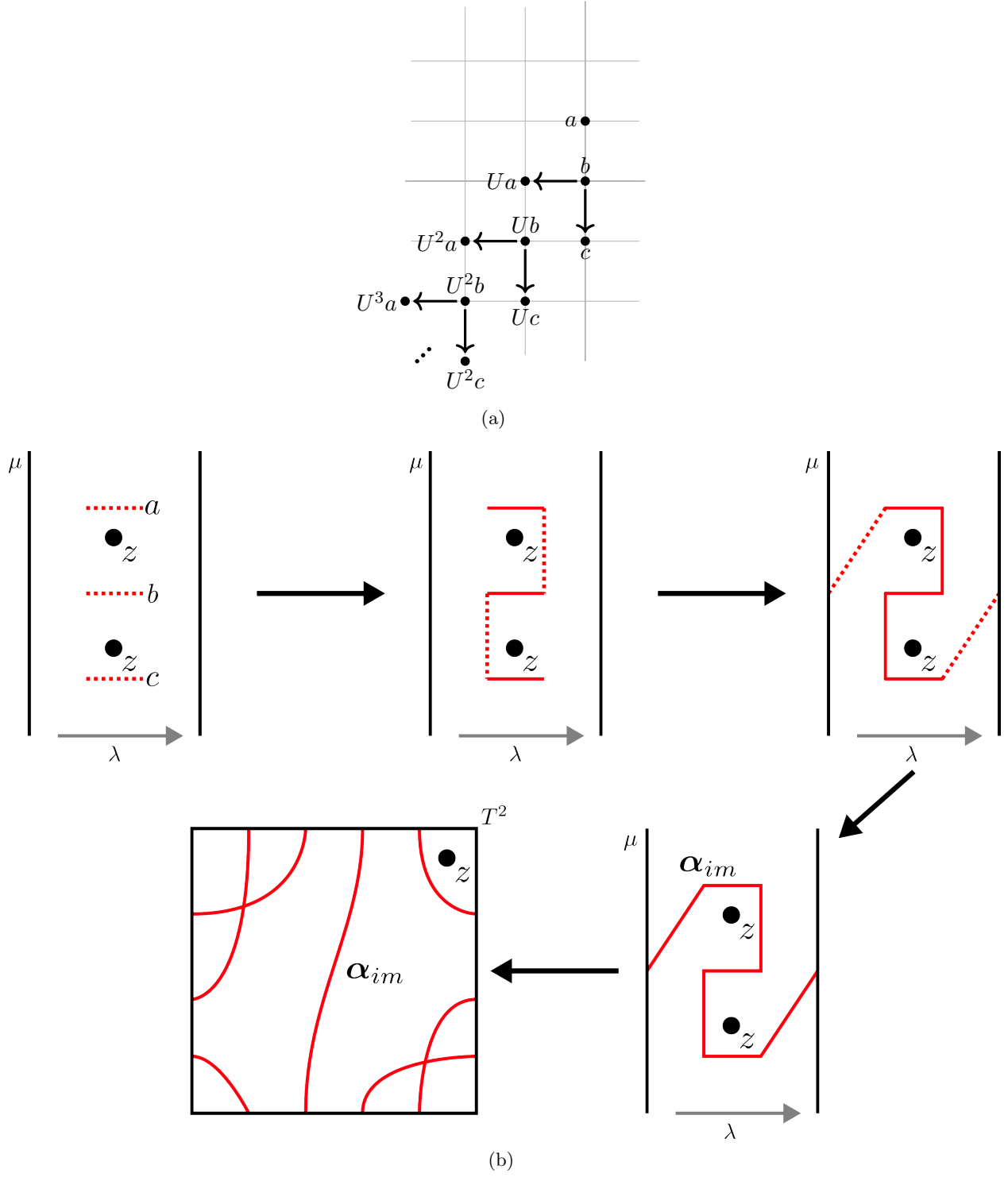
\begin{figure}
    \centering
    \subfigure[]{
    \begin{tikzpicture}
	\begin{scope}[thin, gray]
		\draw [-] (-3, 0) -- (1, 0);
		\draw [-] (0, -3) -- (0, 3);
	\end{scope}
	\draw[step=1, black!30!white, very thin] (-2.9, -2.9) grid (0.9, 2.9);
	
	\filldraw (0, 1) circle (2pt) node[] (a2){};

	\node [left] at (a2) {$a$};
	
	\filldraw (-1, 0) circle (2pt) node[] (a3){};
	\filldraw (0, 0) circle (2pt) node[] (b3){};
	\filldraw (0, -1) circle (2pt) node[] (c3){};
	\draw [very thick, <-] (a3) -- (b3);
	\draw [very thick, <-] (c3) -- (b3);
	\node [left] at (a3) {$Ua$};
	\node [above] at (b3) {$b$};
	\node [below] at (c3) {$c$};

	\filldraw (-2, -1) circle (2pt) node[] (a4){};
	\filldraw (-1, -1) circle (2pt) node[] (b4){};
	\filldraw (-1, -2) circle (2pt) node[] (c4){};
	\draw [very thick, <-] (a4) -- (b4);
	\draw [very thick, <-] (c4) -- (b4);
	\node [left] at (a4) {$U^2a$};
	\node [above] at (b4) {$Ub$};
	\node [below] at (c4) {$Uc$};
	
	\filldraw (-3, -2) circle (2pt) node[] (a5){};
	\filldraw (-2, -2) circle (2pt) node[] (b5){};
    \filldraw (-2,-3) circle (2pt) node[] (c5){};
	\draw [very thick, <-] (a5) -- (b5);
    \draw [very thick, ->] (b5) -- (c5);
	\node [left] at (a5) {$U^3a$};
	\node [above] at (b5) {$U^2b$};
    \node [below] at (c5) {$U^{2}c$};
	
	\filldraw (-2.6, -2.8) circle (0.8pt);
	\filldraw (-2.7, -2.9) circle (0.8pt);
	\filldraw (-2.8, -3) circle (0.8pt);
    \end{tikzpicture}
    }
    \subfigure[]{
    \includegraphics[width=\linewidth]{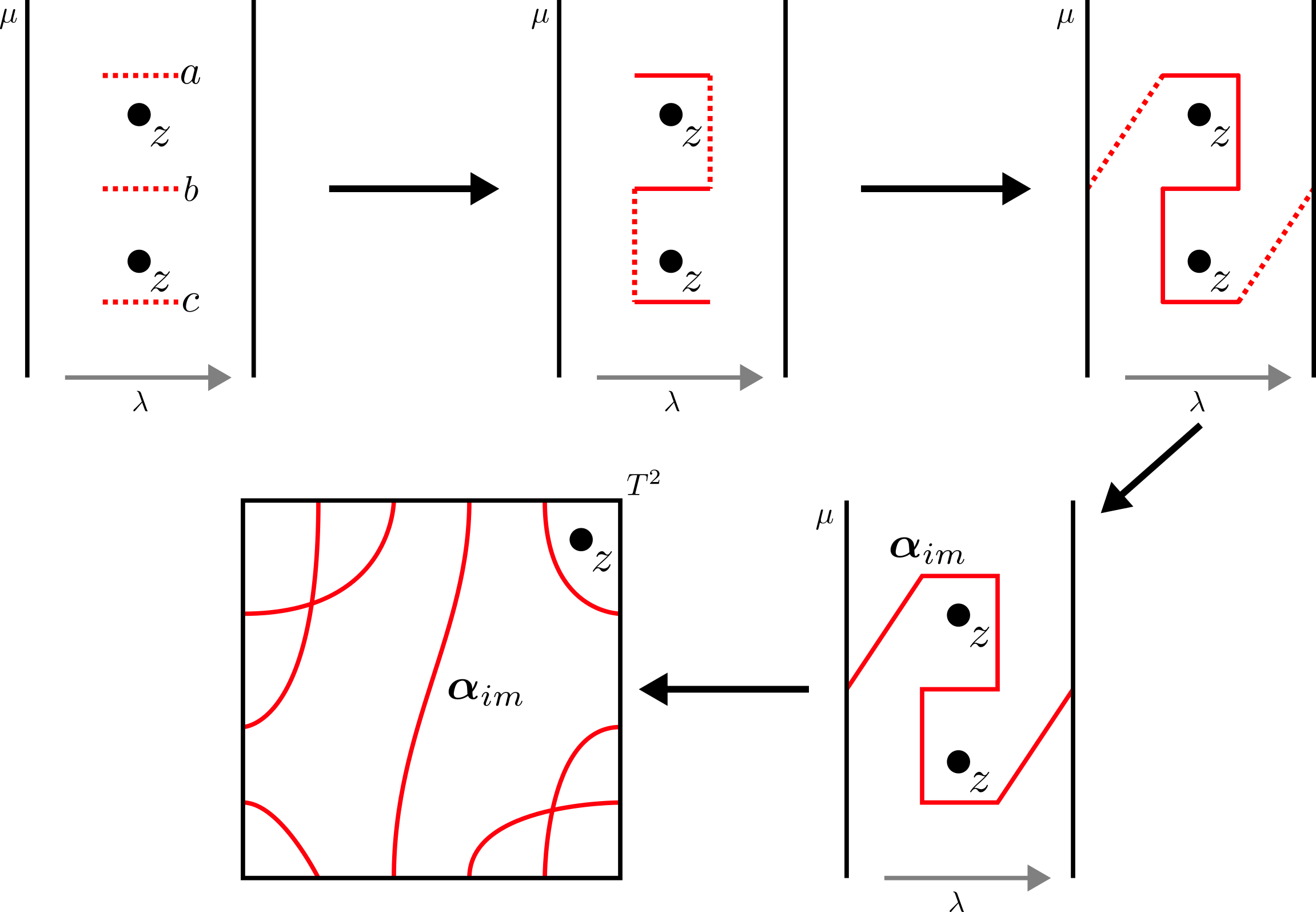}
    }
    \caption{(a) $\CFKm(T_{2,3})$. (b) Each step of the construction in Proposition \ref{prop:makethecurve}, with a projection to the marked torus on the bottom left.}
    \label{fig:immersedexample}
\end{figure}

\subsection{Doubly-pointed Bordered Heegaard Diagrams}

The second piece we need is a doubly-pointed bordered Heegaard diagram. A doubly-pointed bordered Heegaard diagram for a pair $(Y,K)$ is a bordered Heegaard diagram with extra basepoints, so that it describes a knot $K$ in the bordered manifold $Y$. In our case, $Y$ is a solid torus, and $K$ is the pattern knot. We start with the following information to encode a \textit{bordered 3-manifold} $(Y,\cZ,\phi)$, where $\cZ$ and $\phi$ are auxiliary data specifying a parametrization of the boundary of $Y$.

\noindent
\begin{minipage}[r]{\textwidth}
\begin{definition}[Bordered Heegaard Diagram]\label{def:borderedhd}
    A \textit{bordered Heegaard diagram} for a smooth 3-manifold $Y$ with boundary is a tuple $(\Sbar, \alphas, \betas, w)$ where
    \begin{itemize}
        \item $\Sbar$ is a genus $g$ surface with a single boundary component,
        \item $\betas$ is a collection of $g$ pairwise disjoint properly embedded simple closed curves in the interior of $\Sbar$ which are linearly independent in $H_{1}(\Sbar;\Z)$,
        \item $\alphas$ is a collection of $g-k$ pairwise properly disjoint embedded simple closed curves $\alphas^{c}:=\{\alpha^{c}_{1},...,\alpha^{c}_{g-k}\}$ in the interior of $\Sbar$ and $2k$ pairwise disjoint properly embedded arcs $\alphas^{a}:=\{\alpha^{a}_{1},...,\alpha^{a}_{2k}\}$ in $\Sbar$ with transverse intersection with $\del \Sbar$, and
        \item $w$ is a point on $\del\Sbar\setminus (\alphas\cap\del\Sbar)$.
    \end{itemize}
\end{definition}
\end{minipage}
\indent

\begin{figure}
    \centering
    \includegraphics[width=0.6\linewidth]{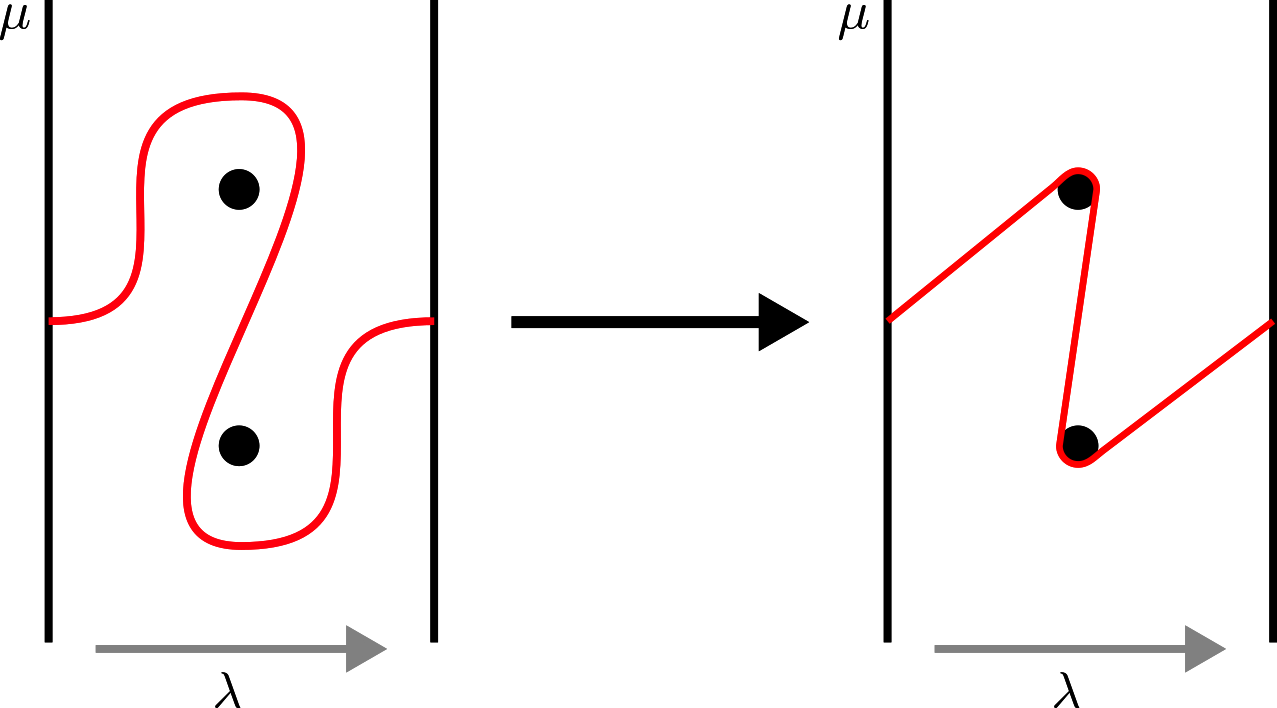}
    \caption{Left: The immersed curve for $T_{2,3}.$ Right: The pegboard for $T_{2,3}.$}
    \label{fig:pegboard}
\end{figure}

Given a bordered Heegaard diagram, we can reconstruct the 3-manifold in the usual way, and recover the parametrization of its boundary using the $\alphas$ arcs. For further details, see \cite{lipshitz2018bordered}.

As we have come to expect, we have the following theorem.

\begin{theorem}[\cite{lipshitz2018bordered}]
    Any bordered 3-manifold can be represented by some bordered Heegaard diagram.
\end{theorem}

We now incorporate the information of a knot in a bordered manifold by adding an extra basepoint to the bordered diagram.

\begin{definition}[Doubly-pointed bordered Heegaard diagram]
    A \textit{doubly-pointed bordered Heegaard diagram} for $Y$ compatible with a knot $K\into Y$ is a tuple $(\Sbar, \alphas, \betas, w, z)$ where
    \begin{itemize}
        \item $(\Sbar, \alphas, \betas, w)$ is a bordered Heegaard diagram for $Y$ as in Definition \ref{def:borderedhd} and
        \item $z$ is a basepoint along with $w$ in $\Sbar\setminus(\alpha^{c}\cup\betas)$ such that $K$ can be recovered from $w$ and $z$.
    \end{itemize}
\end{definition}

As with the traditional case, every knot in a bordered 3-manifold can be realized by a doubly-pointed bordered Heegaard diagram. Figure \ref{fig:whiteheaddiagram} shows a doubly-pointed Heegaard diagram for the Whitehead pattern.

\begin{figure}
    \centering
    \includegraphics{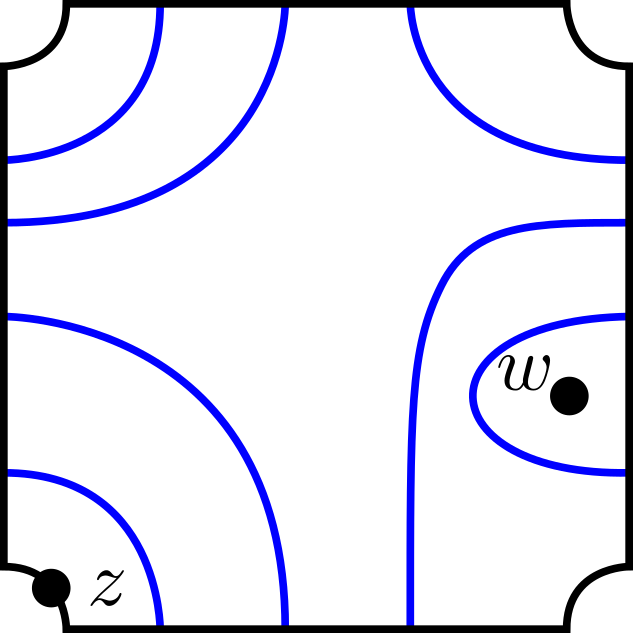}
    \caption{A genus 1 Heegaard diagram for the Whitehead pattern, $D_{+}$. We think of the two arcs $\{\alpha_{1}^{a}, \alpha_{2}^{a}\}$ as $\mu$ and $\lambda$, the two sides of $\Sbar$, the punctured torus ($g=1$). They intersect $\del \Sbar$ as required.}
    \label{fig:whiteheaddiagram}
\end{figure}

\subsection{Immersed Heegaard Floer homology}

We introduce the immersed Heegaard Floer homology from \cite{chenhanselman} to pair the two pieces above to compute the knot Floer homology of the satellite knot.

\begin{definition}\label{def:immersedhd}
    An \textit{immersed doubly-pointed Heegaard diagram} is a tuple $\Hee=(\Sg, \alphas, \betas, w, z)$ where
    \begin{itemize}
        \item $\Sg$ is a closed oriented genus $g$ surface,
        \item $\alphas=\{\alpha_{1},...,\alpha_{g-1}, \alphas_{g}\}$ is a collection of curves in $\Sg$ where $\{\alpha_{1},...,\alpha_{g-1}\}$ are embedded and disjoint, $\alphas_{g}=\{\alpha_{g}^{1},...,\alpha_{g}^{n}\}$ is a collection of immersed curves decorated with local systems for which $\alpha_{g}^{1}$ has the trivial local system, $\{\alpha_{1},...,\alpha_{g-1},\alpha_{g}^{1}\}$ are linearly independent in $H_{1}(\Sg;\Z)$, and each $\alpha_{g}^{i}$ is trivial in $H_{1}(\Sg;\Z)/(\alpha_{1},...,\alpha_{g-1})$ for $i>1$,
        \item $\betas=\{\beta_{1},...,\beta_{g}\}$ is a collection of embedded disjoint curves in $\Sg$ which are linearly independent in $H_{1}(\Sg;\Z)$, and
        \item $w$ and $z$ are basepoints on $\Sg$ lying in the same component of $\Sg\setminus\alphas$ and in the same component of $\Sg\setminus\betas$.
    \end{itemize}
\end{definition}

The local systems we consider in this paper are all trivial. Also, we often denote $\alphas_{g}$ by $\alphas_{im}$ for \emph{immersed} alpha curves, following conventions in \cite{chenhanselman}. As usual, there are admissibility conditions on these kinds of Heegaard diagrams. See \cite{chenhanselman} for details.

\begin{definition}\label{def:CFKbordered}
    Given an immersed doubly-pointed Heegaard diagram $\cH=(\Sg,\alphas,\betas,w, z)$, the \textit{``$UV$ equals zero'' knot Floer complex} of $\cH$ is freely generated over $\cR=\F[U,V]/(UV)$ by all intersection points $\x\in\Ta\cap\Tb$, denoted $$CFK_{\cR}(\cH):=\bigoplus_{\x\in\Ta\cap\Tb}\cR\langle\x\rangle,$$ with boundary map defined on generators by $$\partial_{\cR}(\x):=\sum_{y\in \mathbb{T}_{\alpha}\bigcap\mathbb{T}_{\beta}}\sum_{\substack{\phi\in\pi_{2}(x,y)\\ \mu(\phi)=1}}\#\widehat{\mathcal{M}}(\phi)U^{n_{w}(\phi)}V^{n_{z}(\phi)}\y.$$
\end{definition}

Working over the ring $\cR=\F[U,V]/(UV)$, often called the ``$UV$ equals zero'' ring is useful in eliminating pesky arrows in the chain complex. The complex $\CFKR(\cH)$ supports a bigrading by $(\gru,\grv)$ or by $(\gru, A)$. The following theorems guarantee an invariant chain complex.

\begin{theorem}[\cite{chenhanselman}]
    The complex $(\CFKR(\cH),\del_{\cR})$ is a chain complex, i.e. $\del_{\cR}^{2}=0$.
\end{theorem}

\begin{theorem}[\cite{chenhanselman}]
    The bigraded chain homotopy type of $(\CFKR(\cH),\del_{\cR})$ is invariant under isotopies of the $\alphas$ curves and $\betas$ curves, handleslides, and (de)stabilization of the Heegaard diagram, $\cH$.
\end{theorem}

\subsection{Pairing Theorems}

We can pair a doubly-pointed bordered Heegaard diagram $\cH$ as in Definition \ref{def:borderedhd} and an immersed curve $\alphas_{im}$ as described in Section \ref{sec:immersedcurves} by gluing to obtain a doubly-pointed immersed Heegaard diagram $\cH(\alphas_{im})$ as in Definition \ref{def:immersedhd}. While a more detailed description can be found in \cite{chenhanselman}, essentially, the two diagrams are glued along their common boundary, thought of as ``filling in'' the bordered Heegaard diagram with the immersed curve. After performing certain isotopies, the resulting doubly-pointed immersed Heegaard diagram looks like a superposition of the bordered Heegaard diagram and the immersed curve. Figure \ref{fig:pairing} illustrates this process. We can remove immersed points by looking at the curves in various lifts of the torus, such as the infinite strip. To see why this is useful, it is best to introduce the necessary pairing theorem, as it is one of the primary theorems applied to prove Theorem \ref{thm:deepslice}.

\begin{figure}
    \centering
    \includegraphics[scale=0.65]{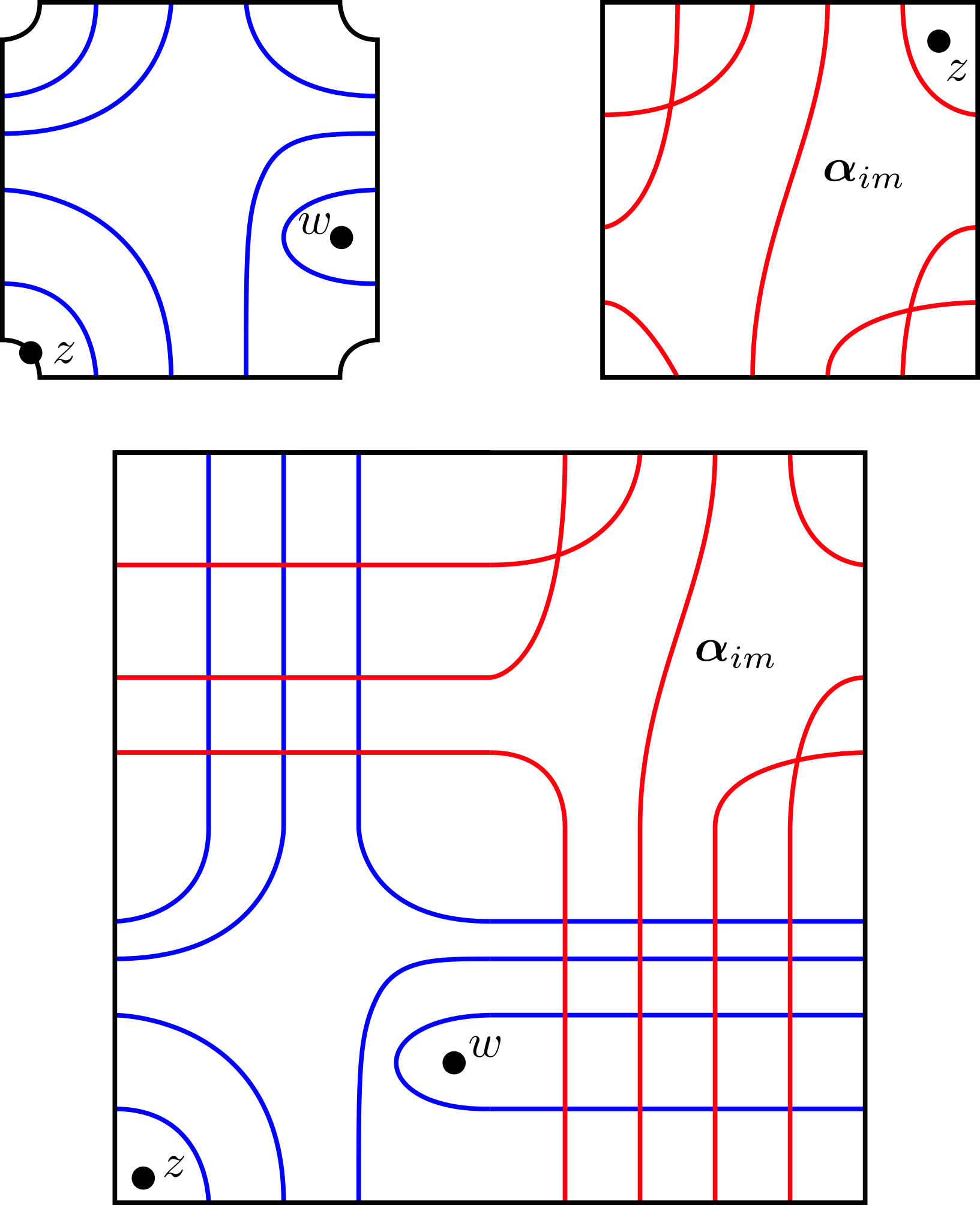}
    \caption{Top left: A bordered Heegaard diagram. Top right: An immersed curve in the marked torus. Bottom: The immersed Heegaard diagram obtained by gluing.}
    \label{fig:pairing}
\end{figure}

\begin{theorem}[\cite{chenhanselman}]\label{thm:pairing1}
    Let $\cH$ be a doubly-pointed bordered Heegaard diagram for a pattern knot $P\subset S^{1}\times \D^{2}$, and let $\alphas_{K}$ be the immersed multicurve associated to a companion knot $K$. Let $\cH(\alphas_{K})$ be the immersed doubly-pointed Heegaard diagram obtained by pairing $\cH$ with $\alphas_{K}$. Then the complex $\CFKR(\cH(\alphas_{K}))$ is bigraded chain homotopy equivalent to the knot Floer complex of the satellite knot $P(K)$ over $\cR$, where $\cR=\F[U,V]/(UV)$.
\end{theorem}

Theorem \ref{thm:pairing1} is a generalization of a theorem in \cite{chen2019pattern}, presented below. In \cite{chenhanselman}, they remark that Theorem \ref{thm:pairing1} is particularly useful when the pattern knot is a \textit{(1,1) pattern}, which means that it has a genus 1 doubly-pointed bordered Heegaard diagram. This is because the resulting immersed Heegaard diagram is also genus 1, so it is easy to extract $\CFKR(\cH(\alphas_{im}))$ even when the curves self-intersect. In fact, the process is entirely combinatorial as we will show in Section \ref{sec:immersedexamples}. The earlier theorem, while less general, is still useful to present here as its proof more carefully specifies the homeomorphism of the pairing of the knot compliment and the pattern torus:

\begin{theorem}[\cite{chen2019pattern}]\label{thm:pairing2}
    Let $P\subset S^{1}\times \D^{2}$ be a $(1,1)$-pattern knot and $K$ in $S^{3}$ a companion. Let $\alphas_{K}\subset (\partial \sthree\setminus\nu(K))$ be the immersed curve for $K$, and let $\cH$ be a genus 1 bordered Heegaard diagram for $P$, thought of as curves and basepoints in $\partial S^{1}\times \D^{2}$. Let $h:\partial (\sthree\setminus\nu(K))\to \partial (S^{1}\times \D^{2})$ be an orientation preserving homeomorphism such that
    \begin{itemize}
        \item $h$ identifies the meridian and Seifert longitude of $K$ with $\mu$ and $\lambda$ respectively;
        \item $h$ maps the $z$ basepoint for $\alphas_{K}$ to the $z$ basepoint for $\cH$;
        \item there is a regular neighborhood $U\subset \partial (S^{1}\times \D^{2})$ of $z$ such that $U\cap(\lambda \cup \mu)=\emptyset$ and $U\cup h(\alphas_{K})=\emptyset$.
    \end{itemize}
    Suppose $\alphas_{K}$ is connected. Then there is a grading-preserving isomorphism of chain complexes $$\CFKhat(\cH(\alphas_{K}))\cong \CFKhat(S^{3}, P(K)).$$
\end{theorem}

The goal is to adapt this theorem to manifolds other than $\sthree$, and extend the chain homotopy equivalence to the $UV=0$ complex $\CFKR$ rather than simply $\CFKhat$.

\begin{theorem}\label{thm:pairing3}
    Let $\cH$ be a doubly-pointed bordered Heegaard diagram for a pattern knot $P\subset S^{1}\times \D^{2}$, and let $\alphas_{K}$ be the immersed multicurve associated to a companion knot $K$. Let $\cH(\alphas_{K})$ be the immersed doubly-pointed Heegaard diagram obtained by pairing $\cH$ with $\alphas_{K}$ using a framing change in accordance with $1/n$ surgery on $K$. Then the knot Floer complex $\CFKR(\cH(\alphas_{K}))$ is bigraded chain homotopy equivalent to the knot Floer complex of the satellite knot $P(\muk)$ in $\dW$ over $\cR=\F[U,V]/(UV)$.
\end{theorem}

\begin{proof}
    The proof of Theorem \ref{thm:pairing1} uses an \textit{arced bordered Heegaard diagram}, which is a version of a bordered Heegaard diagram for a manifold with two boundary components. The manifold in question is $S^{1}\times\D^{2}\setminus\nu(P)$, the complement of the pattern knot in the solid torus. In the proof of Theorem \ref{thm:pairing1}, the parametrization of the outer boundary is the usual meridian-longitude parametrization and the inner boundary is parametrized by the meridian of $P$ and a longitude of $P$. Here, we parametrize the outer boundary instead with a framing change given by the $1/n$ surgery performed along $K$. That is, a homeomorphism $\Phi_{n}:\del_{\mathrm{outer}} \big((S^{1}\times\D^{2})\setminus\nu(P)\big)\to \del (\sthree\setminus\nu(K))$, given by how it acts on the homology, $$\Phi_{n}^{*}=\begin{bmatrix}
        1 & 0\\
        n & -1
    \end{bmatrix}.$$ Now, when pairing $\cH$ with $\alphas_{K}$, we simply take this map into account by adding Dehn twists in the torus for $\cH$ to skew the diagram to slope $-1/n$ (we can skew either diagram using Dehn twists, but we prefer to look at covering spaces which maintain the basis already in place for $\alphas_{K}$ rather than $\cH=(\Sbar, \alphas, \betas, w, z)$, so the linear map on $H_{1}(\Sbar)=<\mu, \lambda>$ is the inverse of $\Phi_{n}^{*}$). The remainder of the proof is unchanged from that of Theorem \ref{thm:pairing1}.
\end{proof}

\subsection{Computing $\tau_{\alpha}$ from Immersed Curves}
\cite{chen2019pattern} introduces a nice way to read the Alexander gradings in an immersed Heegaard diagram of genus 1 and determine which generator survives to the infinity page. This allows for easy computation of $\tau(K)$. The calculus introduces to the diagram \textit{A-buoys} attached to the $\betas$ curves. These A-buoys are small arrows which record the change in Alexander filtration level between two generators as we isotope away disks only crossing the $z$ basepoint. Isotoping disks away is akin to canceling components of the differential on the filtered complex. The difference in Alexander filtration level between two intersection points indicates the lengths of differentials that we cancel, which corresponds to ``turning the page'' on the spectral sequence converging to the Heegaard Floer homology of the underlying 3-manifold, as described in Section \ref{chap:HF}.

So, in practice, we perform isotopies to cancel differentials of filtration length one until we no longer can, recording the filtration change using A-buoys. We can then iterate this process for length two, or three, to see further pages in the spectral sequence. For a knot in $\sthree$, there will eventually be only one remaining intersection point, whose Alexander grading is $\tau(K)$. In general, for a knot in a 3-manifold other than $\sthree$, such as a 3-manifold obtained by surgery, more than one intersection point may remain, since the 3-manifold might have $\mathrm{rk}\HFhat(Y)>1$. Using the Alexander gradings of these remaining generators we can compute $\tau_{\alpha}$.

If there exists a pair of remaining generators, after the above process, with different Alexander grading, then at least one of them is nonzero. Corollary \ref{cor:taualpha} implies that obstructs shallow sliceness of the corresponding knot.

\subsection{Examples}\label{sec:immersedexamples}
Now that the theory is introduced, we move to illuminating examples. We return our attention to the $(1,1)$ pattern knot of focus, the Whitehead double, $D_{+}$, depicted in Figure \ref{fig:satelliteop}. A doubly-pointed bordered Heegaard diagram for $D_{+}$, denoted simply by $\cH$ in this section, can be seen in Figure \ref{fig:whiteheaddiagram}. 

\subsubsection{Example 1: $D_{+}(T_{2,3})$}\label{sec:Trefoil}

We begin with the right handed trefoil, whose knot Floer complex is relatively simple. Its knot Floer complex over $\cR$ is in Figure \ref{fig:immersedexample}. Following the procedure in Proposition \ref{prop:makethecurve}, we obtain the immersed curve for the trefoil shown also in Figure \ref{fig:immersedexample}.

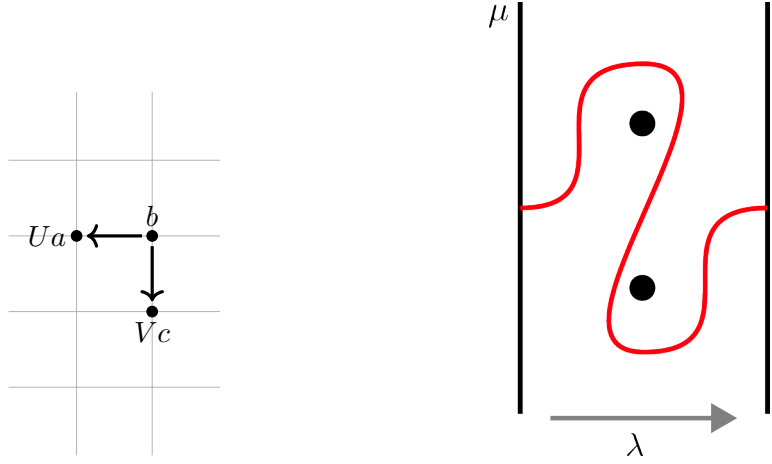
\begin{figure}
    \centering
    \subfigure{
    \begin{tikzpicture}
	\draw[step=1, black!30!white, very thin] (-1.9, -2.9) grid (0.9, 1.9);
	\filldraw (-1, 0) circle (2pt) node[] (a1){};
	\filldraw (0, 0) circle (2pt) node[] (b1){};
	\filldraw (0, -1) circle (2pt) node[] (c1){};
	\draw [very thick, <-] (a1) -- (b1);
	\draw [very thick, <-] (c1) -- (b1);
	\node [left] at (a1) {$Ua$};
	\node [above] at (b1) {$b$};
	\node [below] at (c1) {$Vc$};
    \end{tikzpicture}
    }
    \hspace{1.2in}
    \subfigure{
    \includegraphics{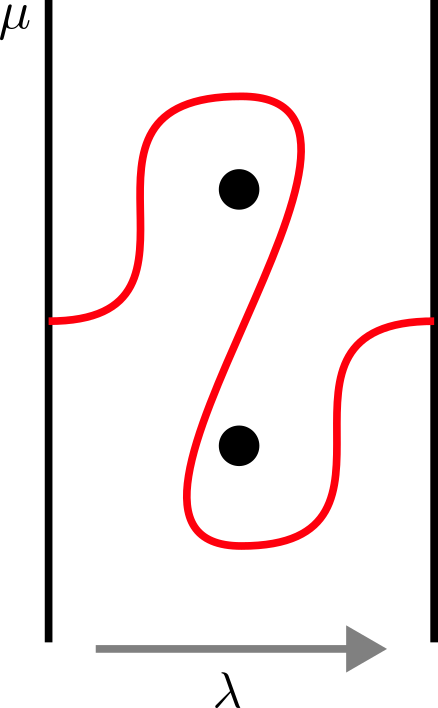}
    }
    \caption{Left: A shorthand representation of the complex $\CFKR(T_{2,3})$, the right handed trefoil. Right: The immersed curve arising from the complex on the left.}
    \label{fig:trefoilimmersedcurve}
\end{figure}

We establish $\tau(D_{+}(T_{2,3}))$, by passing to an immersed Heegaard diagram by pairing $\alphas_{T_{2,3}}$ and $\cH$, as in Figure \ref{fig:WhiteheadTrefoil}. To determine $\tau(D_{+}(T_{2,3}))$ from this immersed Heegaard diagram, we employ two steps. First, the Alexander grading of each intersection point in the diagram are determined by looking at their relative Alexander grading given by Whitney disks between them. Second, we symmetrize the gradings so that the top-most and bottom-most are only opposite in sign. This establishes the absolute Alexander grading of each generator of $\CFKR(D_{+}(T_{2,3}))$. The gradings are as follows: \begin{align*}
    \x && A(\x)\\
    4,7,11,14 && 1\\
    1,3,6,8,10,13,15 && 0\\
    2,5,9,12 && -1
\end{align*} Then $\tau(D_{+}(T_{2,3}))$ is found by following the A-buoy calculus in \cite{chen2019pattern}, canceling differentials of length one, then length two, and so on, until a single generator remains, corresponding to the single generator of $\HFhat(\sthree)$ whose Alexander grading is $\tau(D_{+}(T_{2,3}))$. This manipulation of the immersed Heegaard diagram is shown in Figure \ref{fig:D(T23)Abuoys}. We see that the only generator surviving after pulling $\betas$ straight is $x_{7}$, which has $A(x_{7})=1.$ Thus, $\tau(D_{+}(T_{2,3}))=1$. Consequently, $D_{+}(T_{2,3})$ is not slice.

\begin{figure}
    \centering
    \includegraphics{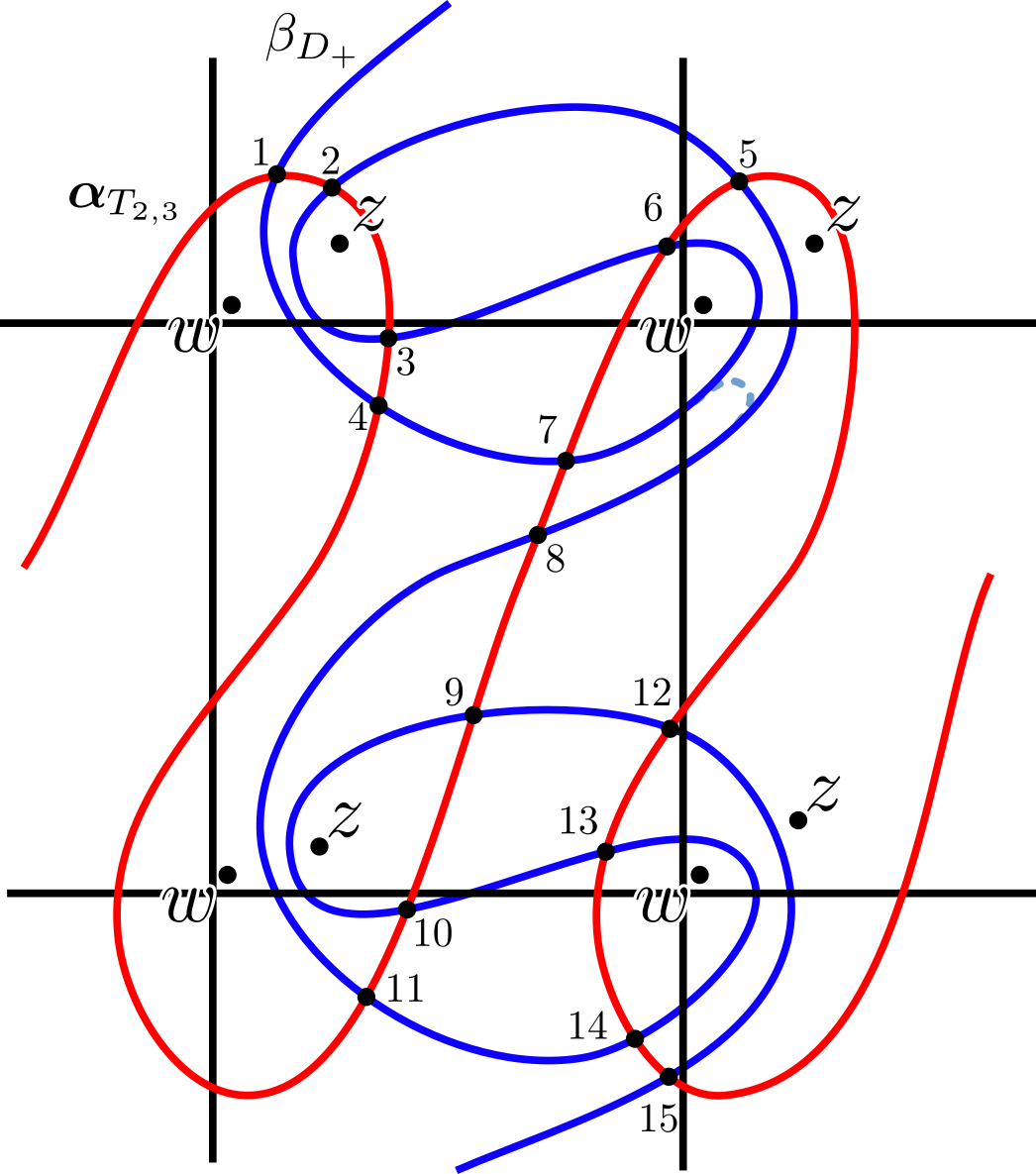}
    \caption{A lift of an immersed Heegaard diagram for the Whitehead double of the right handed trefoil in $\sthree$.}
    \label{fig:WhiteheadTrefoil}
\end{figure}

\begin{figure}
    \centering
    \includegraphics[width=\linewidth]{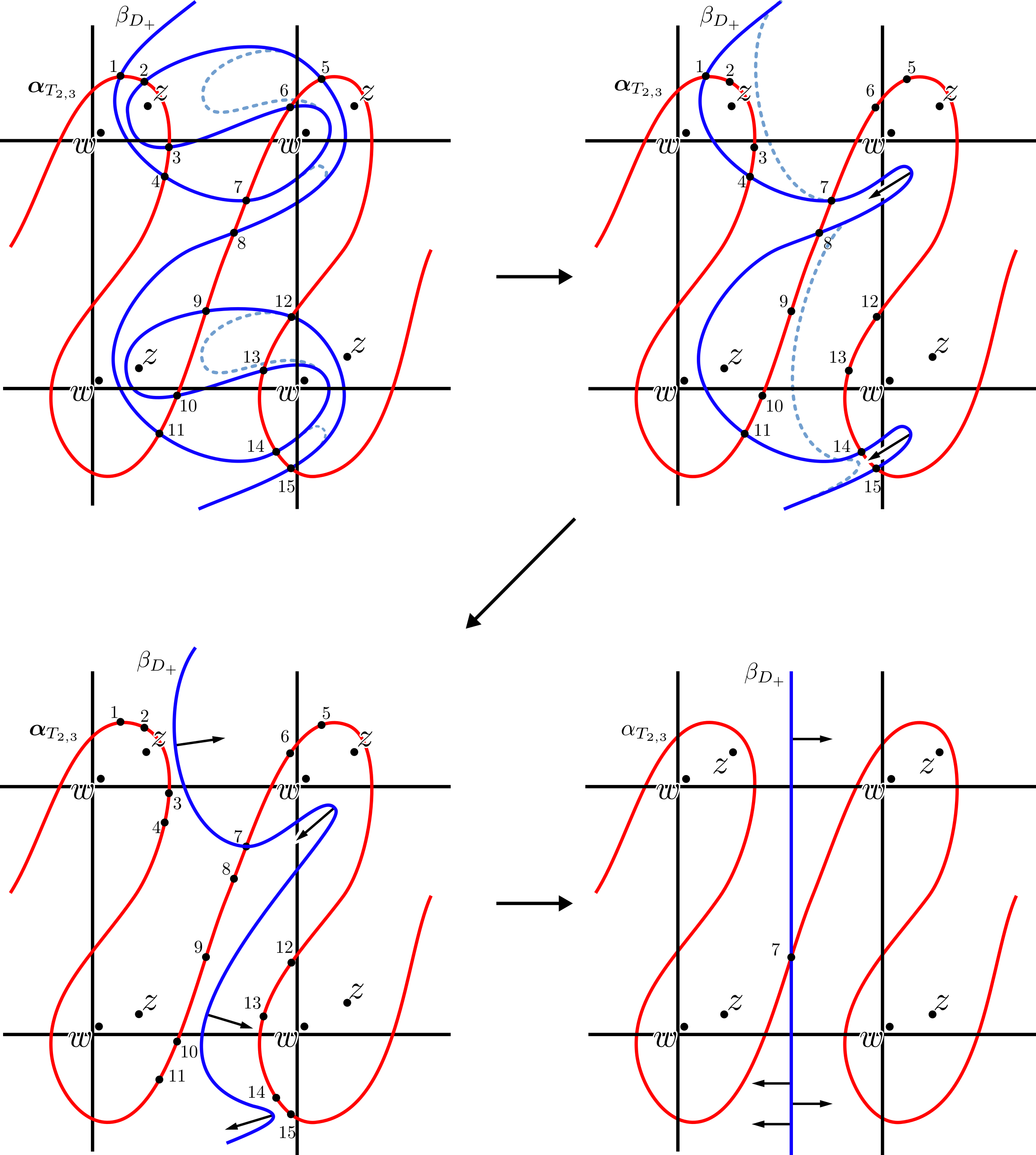}
    \caption{Cancellation of length one differentials. $\beta_{d_{+}}$ is essentially pulled tight across the $z$ basepoints.}
    \label{fig:D(T23)Abuoys}
\end{figure}

\subsubsection{Example 2: $D_{+}\big((4_{1})^{*}_{1/2}\big)$}\label{sec:figureeight}

We now turn our attention to an example of the Whitehead double of the dual knot to $1/n$ surgery for the figure eight knot, $4_{1},$ and $n=2$. Since the surgered 3-manifold is no longer $\sthree$, the expectation is that there will be more than one generator $\alpha\in\HFhat(S^{3}_{1/2}(4_{1}))$ with which to compute $\tau_{\alpha}$. The immersed curve $\alphas_{4_{1}}$ is shown in Figure \ref{fig:figureeightimmersedcurve}. Combining $\alpha_{4_{1}}$ with $\cH$ as in Figure \ref{fig:whiteheaddiagram} using the map $\Phi_{2}^{*}$ as in the proof of Theorem \ref{thm:pairing3}, we arrive at the immersed Heegaard diagram in Figure \ref{fig:final}. In Figure \ref{fig:final}, a convenient lift to a suitable covering space was chosen to remove as many immersed points as possible. This way, it is significantly easier to execute the combinatorics of counting disks and to see the generators. From Figure \ref{fig:final}, we can compute the complex $\CFKR(S^{3}_{1/2}(4_{1}), (4_{1})^{*}_{1/2})$ directly, or simply apply the A-buoy calculus to see the Alexander gradings of the surviving generators after pulling the $\betas$ curve straight, allowing isotopy over $z$ basepoints. In this example, only length one differentials need to be canceled before arriving at Figure \ref{fig:D(41)pulledtight}, where no more isotopies can be made. It is clear that only generators $7,8,9,10,$ and $11$ survive. From the A-buoys and the relative Alexander grading formula, we see that $A(x_{10})\neq A(x_{11})$, so $\tau_{x_{10}}\big(D_{+}((4_{1})^{*}_{1/2})\big)\neq \tau_{x_{11}}\big(D_{+}((4_{1})^{*}_{1/2})\big)$, and, in particular, one of them is nonzero. By Corollary \ref{cor:taualpha}, $D_{+}\big((4_{1})^{*}_{1/2}\big)$ cannot be slice in $S^{3}_{1/2}(4_{1})\times I$. Remarkably, if $4_{1}$ were slice (it is not), then we could conclude that $D_{+}\big((4_{1})^{*}_{1/2}\big)$ is deeply slice in the 4-manifold described in Section \ref{section:topology}, by applying Corollary \ref{cor:dualslice}.

\begin{figure}
    \centering
    \subfigure{
    \begin{tikzpicture}
	\draw[step=1, black!30!white, very thin] (-1.9, -2.9) grid (0.9, 1.9);
	\filldraw (-1, 0) circle (2pt) node[] (a1){};
	\filldraw (0, 0) circle (2pt) node[] (b1){};
	\filldraw (0, -1) circle (2pt) node[] (c1){};
    \filldraw (-1,-1) circle (2pt) node[] (d1){};
    \filldraw (0.2, 0.2) circle (2pt) node[] (e1){}; 
	\draw [very thick, <-] (a1) -- (b1);
	\draw [very thick, <-] (c1) -- (b1);
    \draw [very thick, <-] (d1) -- (c1);
    \draw [very thick, <-] (d1) -- (a1);
	\node [above left] at (a1) {$Ua$};
	\node [below left] at (b1) {$b$};
	\node [below right] at (c1) {$Vc$};
    \node [below left] at (d1) {$UVd$};
    \node [above right] at (e1) {$e$};
    \end{tikzpicture}
    }
    \hspace{1.2in}
    \subfigure{
    \includegraphics{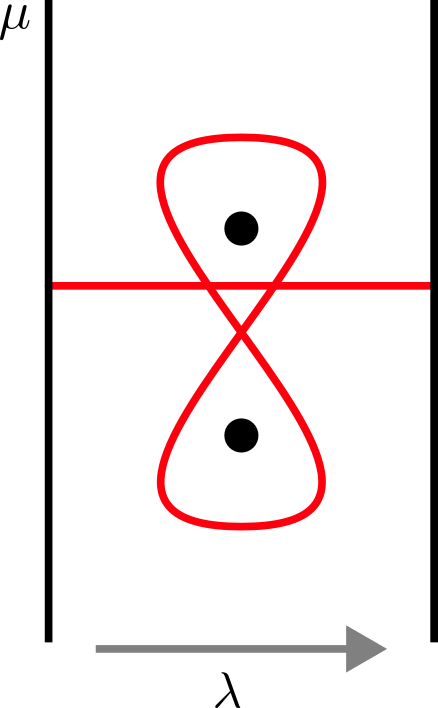}
    }
    \caption{Left: A shorthand representation of the complex $\CFKR(4_{1})$, the figure eight knot. Right: The immersed curve arising from the complex on the left.}
    \label{fig:figureeightimmersedcurve}
\end{figure}

\begin{figure}
    \centering
    \includegraphics{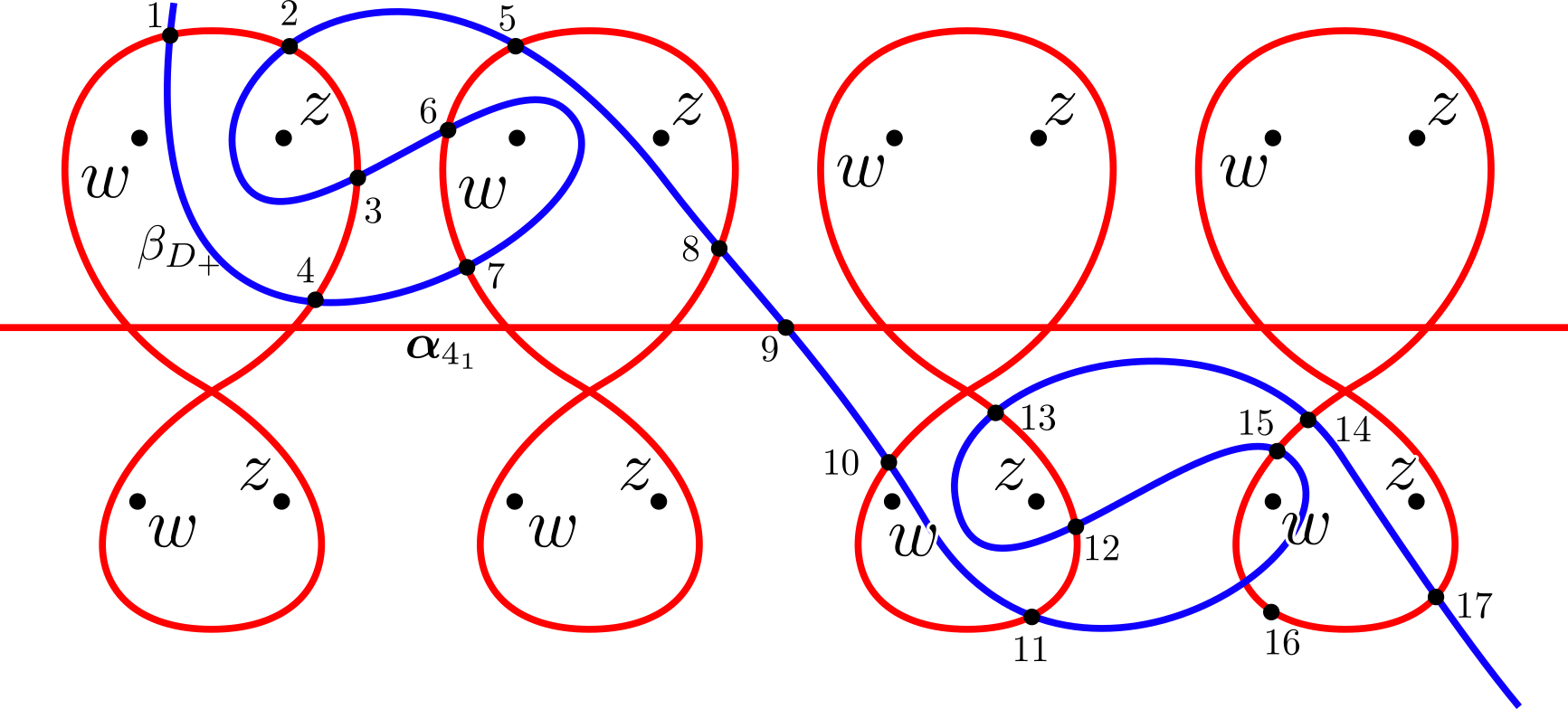}
    \caption{A lift of the immersed Heegaard diagram for the pairing of the Whitehead double pattern with the dual knot to $1/2$ surgery along the figure eight knot, $4_{1}$. Notice the change of framing for the Whitehead pattern corresponding to the $1/2$ surgery on $4_{1}$, giving the $\betas$ curve a $-1/2$ slope.}
    \label{fig:final}
\end{figure}

\begin{figure}
    \centering
    \includegraphics{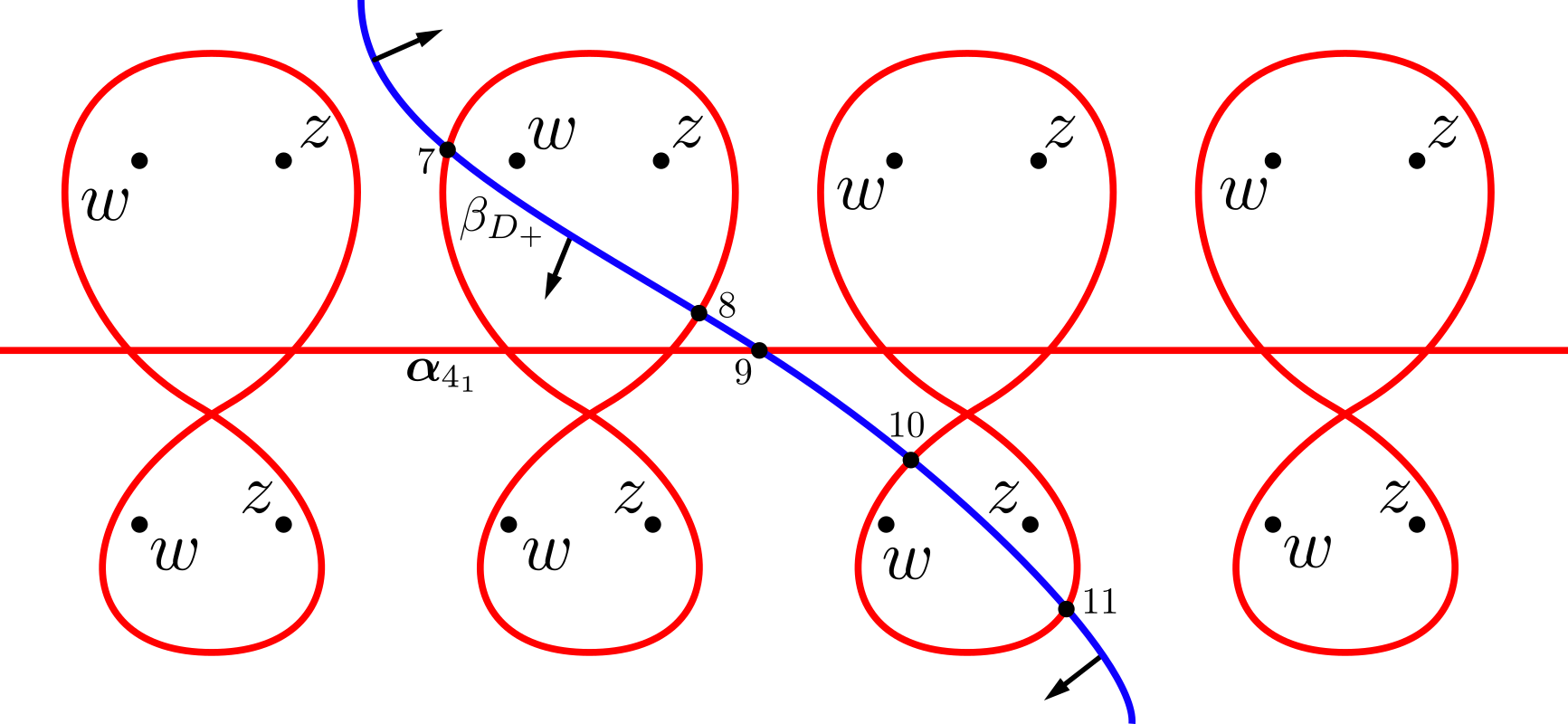}
    \caption{The immersed Heegaard diagram arising from canceling all the length one differentials.}
    \label{fig:D(41)pulledtight}
\end{figure}

\section{Proof of the Main Theorem}\label{sec:maintheorem}

In this section we prove the following theorem.

\begin{theorem}\label{thm:deepslice}
    If a non-trivial knot $K$ is slice in the 4-ball and has an acyclic summand in its knot Floer complex over $\F[U,V]/(UV)$ admitting a simplified basis, then the Whitehead double of the dual knot to $1/n$ surgery along $K$, $\Dmu$, is deeply slice in $W$ (as in Proposition \ref{prop:W}), for $n\in\Z_{>0}$.
\end{theorem}

Note that \cite{klug2021deep} provides a related theorem about the existence of deeply slice knots.

\begin{theorem}\cite[Theorem 3.1]{klug2021deep}\label{theorem:integercase}
Every 2-handlebody $X$ contains a null-homotopic deeply slice knot in its boundary.
\end{theorem}

The proof of Theorem \ref{theorem:integercase} requires two cases depending on whether the boundary has nontrivial $\pi_1$ or not (i.e. if it is or is not $S^{3}$). In our case, since the three manifolds are $1/n$ surgeries on a non-trivial knot, they have non-trivial $\pi_1$, which falls into the second case. In \cite{klug2021deep} they use the topology of the 4 manifold $X$ for the proof. Our proof of Theorem \ref{thm:deepslice} only relies on the computation of the $\tau_\alpha$-invariants.

One noteworthy feature of the computation in Example \ref{sec:figureeight} is that the distinct $\tau_{\alpha}$-invariants came from generators on the closed component of the immersed curve for the knot. For Theorem \ref{thm:deepslice}, we make use of the fact that its immersed curve has a curve corresponding to a box complex, manifesting as a ``figure 8'' shape in the immersed curve. Since Theorem \ref{thm:deepslice} was inspired by a computation for $1/2$ surgery on the knot $6_{1}$, we provide an example following identical steps to those in Section \ref{sec:immersedexamples}.

\begin{example}\label{ex:6_1}
    First, note that $6_{1}\subset \sthree$ is an alternating, slice knot in $B^{4}$. Therefore, $\HFKhat(6_{1})$ is determined by its Alexander polynomial, $\Delta_{6_{1}}(t)=-2t^{-1}+5-2t$, and its signature, $\sigma(6_{1})=0$. By \cite[Theorem 3.1]{ozsvath2003alternating}, $$\HFKhat(\sthree, 6_{1})=\F_{-1}^{2}\oplus\F^{5}_{0}\oplus\F_{1}^{2},$$ where the subscript denotes the Alexander grading. Using the spectral sequence from $\HFKhat(6_{1})$ to $\HFhat(\sthree)$, we can reconstruct the vertical arrows present in the $UV=0$ knot Floer complex, as in Figure \ref{fig:rebuild61}. Now, using the symmetry granted by swapping the roles of $U$ and $V$, we reconstruct horizontal arrows in $\CFKR(6_{1})$ and plot more of the generators in the $(\gru, \grv)$ plane, as in Figure \ref{fig:61complex}.

    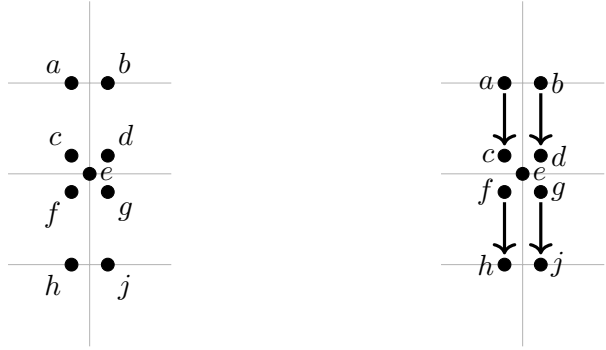
\begin{figure}
        \centering
        \subfigure{
    \begin{tikzpicture}[scale=1.2]
	\draw[step=1, black!30!white, very thin] (-0.9, -1.9) grid (0.9, 1.9);
	\filldraw (-0.2,1) circle (2pt) node[] (a){};
    \filldraw (0.2,1) circle (2pt) node[] (b){};
    \filldraw (-0.2, 0.2) circle (2pt) node[] (c){};
    \filldraw (0.2, 0.2) circle (2pt) node[] (d){};
    \filldraw (0,0) circle (2pt) node[] (e){};
    \filldraw (-0.2, -0.2) circle (2pt) node[] (f){};
    \filldraw (0.2, -0.2) circle (2pt) node[] (g){};
    \filldraw (-0.2, -1) circle (2pt) node[] (h){};
    \filldraw (0.2, -1) circle (2pt) node[] (j){};
	\node [above left] at (a) {$a$};
	\node [above right] at (b) {$b$};
	\node [above left] at (c) {$c$};
    \node [above right] at (d) {$d$};
    \node [right] at (e) {$e$};
    \node [below left] at (f) {$f$};
    \node [below right] at (g) {$g$};
    \node [below left] at (h) {$h$};
    \node [below right] at (j) {$j$};    
    \end{tikzpicture}
    }
    \hspace{1.2in}
        \subfigure{
    \begin{tikzpicture}[scale=1.2]
	\draw[step=1, black!30!white, very thin] (-0.9, -1.9) grid (0.9, 1.9);
	\filldraw (-0.2,1) circle (2pt) node[] (a){};
    \filldraw (0.2,1) circle (2pt) node[] (b){};
    \filldraw (-0.2, 0.2) circle (2pt) node[] (c){};
    \filldraw (0.2, 0.2) circle (2pt) node[] (d){};
    \filldraw (0,0) circle (2pt) node[] (e){};
    \filldraw (-0.2, -0.2) circle (2pt) node[] (f){};
    \filldraw (0.2, -0.2) circle (2pt) node[] (g){};
    \filldraw (-0.2, -1) circle (2pt) node[] (h){};
    \filldraw (0.2, -1) circle (2pt) node[] (j){};
	\draw [very thick, <-] (c) -- (a);
	\draw [very thick, <-] (d) -- (b);
    \draw [very thick, <-] (h) -- (f);
    \draw [very thick, <-] (j) -- (g);
	\node [left] at (a) {$a$};
	\node [right] at (b) {$b$};
	\node [left] at (c) {$c$};
    \node [right] at (d) {$d$};
    \node [right] at (e) {$e$};
    \node [left] at (f) {$f$};
    \node [right] at (g) {$g$};
    \node [left] at (h) {$h$};
    \node [right] at (j) {$j$};
    \end{tikzpicture}
        }
        \caption{Left: $\HFKhat(6_{1})$ arranged by Alexander grading. Right: the location of the vertical arrows in $\CFKR(6_{1})$ so there is only one generator for $\HFhat(\sthree)$.}
        \label{fig:rebuild61}
    \end{figure}

    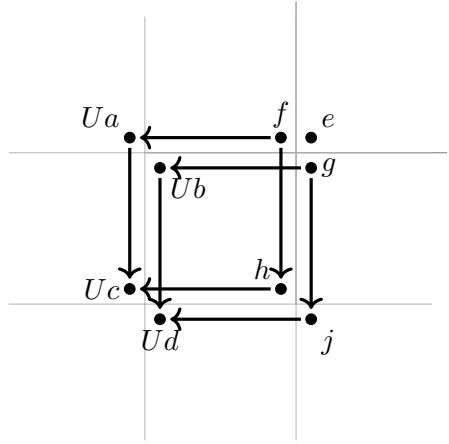
\begin{figure}
        \centering
        \begin{tikzpicture}[scale=2]
	\begin{scope}[thin, gray]
		\draw [-] (-1, 0) -- (1, 0);
		\draw [-] (0, -1) -- (0, 1);
	\end{scope}
	\draw[step=1, black!30!white, very thin] (-1.9, -1.9) grid (0.9, 0.9);
	\filldraw (-1.1,0.1) circle (1pt) node[] (a){};
    \filldraw (-0.9,-0.1) circle (1pt) node[] (b){};
    \filldraw (-1.1, -0.9) circle (1pt) node[] (c){};
    \filldraw (-0.9, -1.1) circle (1pt) node[] (d){};
    \filldraw (0.1,0.1) circle (1pt) node[] (e){};
    \filldraw (-0.1, 0.1) circle (1pt) node[] (f){};
    \filldraw (0.1, -0.1) circle (1pt) node[] (g){};
    \filldraw (-0.1, -0.9) circle (1pt) node[] (h){};
    \filldraw (0.1, -1.1) circle (1pt) node[] (j){};
	\draw [very thick, <-] (c) -- (a);
	\draw [very thick, <-] (d) -- (b);
    \draw [very thick, <-] (h) -- (f);
    \draw [very thick, <-] (j) -- (g);
    \draw [very thick, <-] (d) -- (j);
    \draw [very thick, <-] (c) -- (h);
    \draw [very thick, <-] (a) -- (f);
    \draw [very thick, <-] (b) -- (g);
	\node [above left] at (a) {$Ua$};
	\node [below right] at (b) {$Ub$};
	\node [left] at (c) {$Uc$};
    \node [below] at (d) {$Ud$};
    \node [above right] at (e) {$e$};
    \node [above] at (f) {$f$};
    \node [right] at (g) {$g$};
    \node [above left] at (h) {$h$};
    \node [below right] at (j) {$j$};
    \end{tikzpicture}
        \caption{Adding horizontal arrows in the plane using symmetry along $\gru=\grv$.}
        \label{fig:61complex}
    \end{figure}

    \begin{figure}
        \centering
        \includegraphics{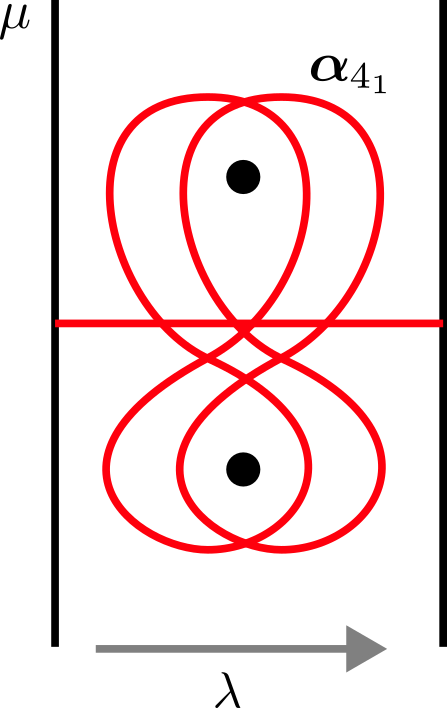}
        \caption{The immersed curve  $\alphas_{6_{1}}$ for $6_{1}$.}
        \label{fig:61immersedcurve}
    \end{figure}

    Now that we have the full information $\CFKR(6_{1})$ in a horizontally and vertically simplified basis, we can construct its immersed curve using the method in Proposition \ref{prop:makethecurve}, shown in Figure \ref{fig:61immersedcurve}. Notably, the curve resembles the curve for $4_{1}$ as in Figure \ref{fig:figureeightimmersedcurve}, with another closed component overlapping the first. In Example \ref{sec:figureeight} only the closed component is necessary to see differing $\tau_{\alpha}$-invariants, therefore we only need to keep track of one ``figure 8'' shape in the pairing diagram with the bordered Heegaard diagram for the Whitehead double, $\cH$. In this case, the local picture near generators $x_{10}$ and $x_{11}$ in Figure \ref{fig:final} looks identical (after isotoping away the other curve components) to the local picture for the pairing of $\cH$ and $\alphas_{6_{1}}$. The same two generators, then, still have $\tau_{x_{10}}\big(D_{+}((6_{1})^{*}_{1/2})\big)\neq \tau_{x_{11}}\big(D_{+}((6_{1})^{*}_{1/2})\big)$, so one is nonzero. By Corollary \ref{cor:taualpha}, $D_{+}\big((6_{1})^{*}_{1/2}\big)$ cannot be slice in $S^{3}_{1/2}(6_{1})\times I$. However, since $6_{1}$ is itself slice, Corollary \ref{cor:dualslice} implies that $D_{+}\big((6_{1})^{*}_{1/2}\big)$ is deeply slice in $W$, the contractible 4-manifold of Proposition \ref{prop:W}.
\end{example}

In finding the desired obstruction, $\tau_\alpha$, we only needed to keep track of the acyclic components of the knot's immersed curve. To prove Theorem \ref{thm:deepslice}, we generalize this concept to see the obstruction by examining how the immersed curve components can interact with the pattern curve once it is ``pulled" tight. To do so, some setup is in order.

\begin{lemma}\label{lem:acyclicbox}
    Given a slice knot with a simplified basis for its knot Floer complex, each generator in the acyclic summand of the complex has exactly one horizontal arrow and exactly one vertical arrow pointing to or away from it.
\end{lemma}

\begin{proof}
    Suppose that there exists a generator $U^kb$ in an acyclic summand with only a horizontal arrow pointing to it or away from it and no other arrows attached to it. Then $b$ is a generator of $\CFhat(\sthree)$ with no arrows pointing to or away from. Hence, it is a generator of $\HFhat(\sthree)$. This is a contradiction since we already restricted $b$ to an acyclic summand of $\CFKm(Y,K)$.

    Suppose instead that $U^kb$ had a singular vertical arrow pointing to it or away from it and no other arrows attached to it. Using the symmetry of $\CFKm(Y,K)$ along the $A=\gru$ axis means there is also a corresponding generator with only a horizontal arrow pointing to it. Then we are again in the above case and arrive at the same contradiction.
\end{proof}

\begin{figure}
     \centering
     \includegraphics[width=0.95\linewidth]{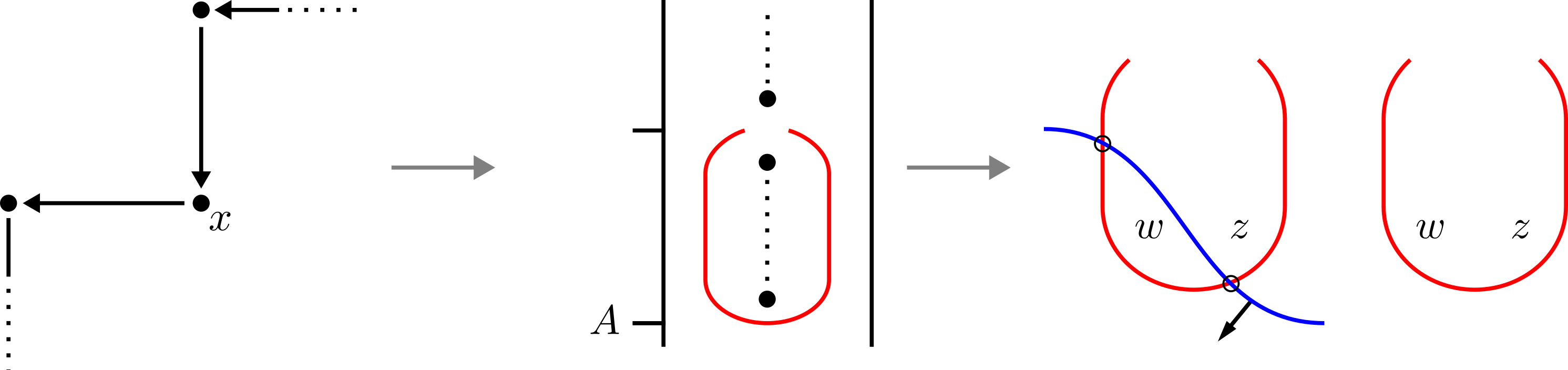}
     \caption{Left: local picture of an acyclic summand near the lowest Alexander grading. Right: corresponding local picture of the immersed Heegaard diagram in a lift of the torus.}
    \label{fig:case1}
\end{figure}

We are free to consider any acyclic summand and any lowest Alexander grading generator in that summand. From Lemma \ref{lem:acyclicbox}, we see that this generator will always have one horizontal arrow pointing away from it, and one vertical arrow pointing to it, as in Figure \ref{fig:case1} (on the left). Since it has lowest Alexander grading, it appears at the bottom of the corresponding immersed curve, and both arrows transfer to segments pointing upwards, which gives the local picture in Figure \ref{fig:case1} (on the right). By symmetry of the knot Floer complex, we can also choose a highest Alexander grading generator corresponding to the lowest one we chose, which has a horizontal arrow pointing to it, and a vertical arrow pointing out of it. The corresponding local immersed curve for the highest Alexander grading generator is obtained from that of the lowest one by a 180-degree rotation around the origin.

\begin{proof}[Proof of Theorem \ref{thm:deepslice}]
    First, note that Corollary \ref{cor:dualslice} guarantees that $\Dmu$ is slice in $W$. It remains to obstruct sliceness in a neighborhood of the boundary of $W$, $\dWxI.$ Doing so involves a process like that of Example \ref{ex:6_1}, where we use an immersed Heegaard diagram to identify generators with differing $\tau_{\alpha}$ invariants. We prove it in different cases, depending on the surgery coefficient.\\
    
    \textbf{Case 1: $n>1$.}\\
    For this case, the generators we use, which survive in the infinite page of the spectral sequence are the same as the ones we use in Example \ref{ex:6_1}. Notice, that before canceling any disks, the intersection points in the diagram on the far right of Figure \ref{fig:case1} are of differing Alexander grading. After pulling tight, the two highlighted points remain, separated by a disk which crosses a $w$ basepoint but not a $z$ basepoint, so it cannot be canceled. Therefore, they have different $\tau_{\alpha}$ invariants (namely, their Alexander gradings). \\
    
    \textbf{Case 2: $n=1$.}
    \begin{figure}
        \centering
        \includegraphics[scale=0.8]{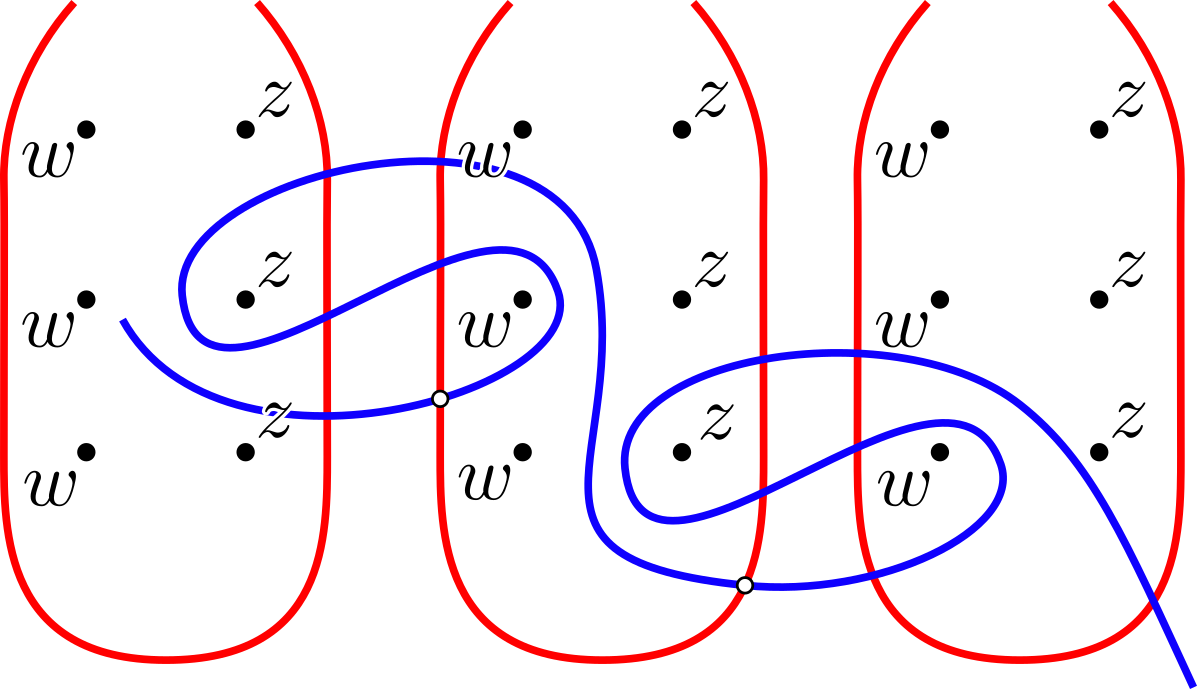}
        \caption{The immersed curve for $n=1$ case.}
        \label{fig:n=1lowest}
    \end{figure}
    
    In Figure \ref{fig:n=1lowest} we see the generators used in the previous case marked by hollow dots, however, in this case these two generators have the same Alexander grading. Instead, we will use the pair of generators from the top Alexander grading generator.

    \begin{figure}
        \centering
        \includegraphics[scale=0.8]{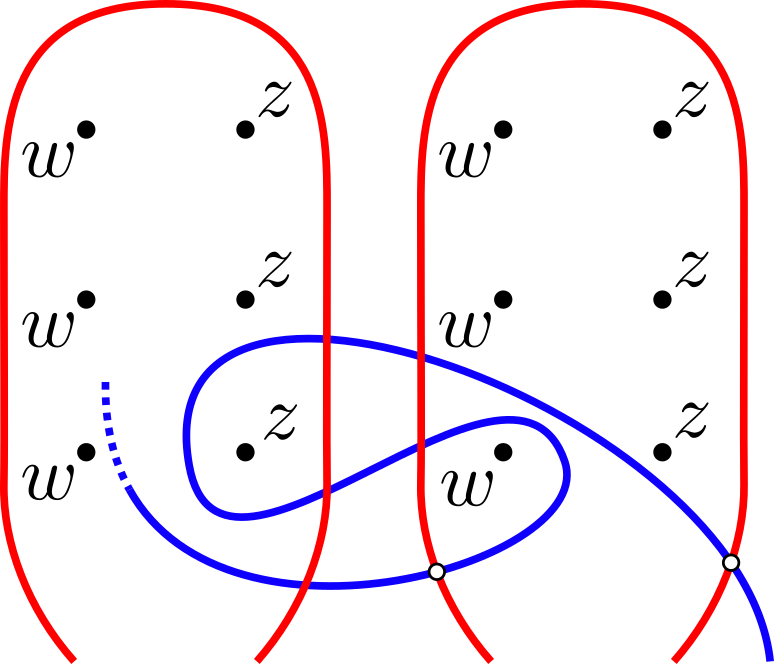}
        \caption{The local immersed curve for the top Alexander grading generator.}
        \label{fig:topAlexander}
    \end{figure}

    In Figure \ref{fig:topAlexander}, we indicate the generators of interest by hollow dots. Given there is an immersion occurring immediately below the local picture in Figure \ref{fig:topAlexander}, we apply the only isotopy available to us, which is sliding the blue curve over the $z$ basepoint enclosed by a Whitney disk to cancel the corresponding length one differential given by that disk. This leaves only the two generators marked with hollow dots. The Alexander gradings of these two generators differ by an odd number, and hence are distinct.
    
    Since we have found pairs of generators with distinct Alexander gradings in each case, it follows that the knot $\Dmu$ is deeply slice in the 4-manifold $W$ constructed in Section \ref{section:topology}.
    \end{proof}

\newpage
\clearpage

\bibliographystyle{amsalpha}
\bibliography{bib}

\end{document}